\documentclass[12pt,twoside]{article}
\usepackage{amsmath,amsthm,amssymb,amscd,ascmac, amsfonts,fancybox}
\usepackage{mathrsfs}
\usepackage{stmaryrd}

\numberwithin{equation}{section}
\usepackage[dvips]{graphicx,color,psfrag}

\makeatletter
\newcommand{\figcaption}[1]{\def\@captype{figure}\caption{#1}}
\newcommand{\tblcaption}[1]{\def\@captype{table}\caption{#1}}
\makeatother

\makeatletter

\ifcase\@ptsize
\def\rpkern{\mathchoice{\kern-1.45em}{\kern-1.11em}{}{}}%
\def\grpkern{\mathchoice{\kern-1.013em}{\kern-0.825em}{}{}}%
\or
\def\rpkern{\mathchoice{\kern-1.44em}{\kern-1.11em}{}{}}%
\def\grpkern{\mathchoice{\kern-1.00em}{\kern-0.81em}{}{}}%
\or
\def\rpkern{\mathchoice{\kern-1.472em}{\kern-1.14em}{}{}}%
\def\grpkern{\mathchoice{\kern-1.00em}{\kern-0.815em}{}{}}%
\fi

\def\minibullet{\mathchoice%
{\raise0.2ex\hbox{$\scriptstyle\bullet$}}%
{\raise0.26ex\hbox{$\scriptscriptstyle\bullet$}}{}{}}
\def\butabullet{\mathchoice%
{\raise0.8ex\hbox{$\scriptstyle\bullet$}{\kern-0.365em}%
\lower0.4ex\hbox{$\scriptstyle\bullet$}}%
{\raise0.75ex\hbox{$\scriptscriptstyle\bullet$}{\kern-0.335em}%
\lower0.25ex\hbox{$\scriptscriptstyle\bullet$}}{}{}}

\def\customprod#1#2%
{\operatorname{\coprod\kern#2{#1}\kern#2\prod}\displaylimits}

\renewcommand{\d}{\delta}

\renewcommand{\Re}{\mathrm{Re}\,}
\renewcommand{\Im}{\mathrm{Im}\,}

\newcommand{\Aut}{\mathop{\mathrm{Aut}}\,}
\newcommand{\tr}{\mathop{\mathrm{tr}}\,}
\renewcommand{\det}{\mathop{\mathrm{det}}\,}

\newtheorem*{multitheorem}{\variable@name}

\theoremstyle{definition}
\newcommand{\variable@name}{Theorem}
\newtheorem*{multiproclaim}{\variable@name}

\theoremstyle{plain}
\newtheorem{thm}{Theorem}[section]
\newtheorem{prop}[thm]{Proposition}
\newtheorem{lem}[thm]{Lemma}
\newtheorem{cor}[thm]{Corollary}

\newtheorem{conj}[thm]{Conjecture}
\newtheorem{prob}[thm]{Problem}

\theoremstyle{definition}
\newtheorem{dfn}[thm]{Definition}
\newtheorem{rmk}[thm]{Remark}

\begin{document}
\title{Multivariate circular Jacobi polynomials}
\author{Genki Shibukawa}
\date{\empty}
\pagestyle{plain}

\maketitle


\begin{abstract}
We introduce a new multivariate orthogonal polynomial which is a 2-parameter deformation of the spherical polynomial by harmonic analysis on symmetric cone. 
This is also regarded as a multivariate analogue of the circular Jacobi polynomial. 
Further, the weight function of its orthogonality relation coincides with the circular Jacobi ensemble defined by Bourgade et al.. 
We also obtain its main properties\,:\,generating function, pseudo differential equation and determinant expression. 
\end{abstract}

\section{Introduction}
The one variable circular Jacobi (orthogonal) polynomials which are named by M.\,Ismail (see \cite{Is}\,p.\,229) are defined by the Gaussian hypergeometric representation as  
\begin{align}
\phi_{m}^{(\alpha)}(x)&:=\frac{(\alpha)_{m}}{m!}{_{2}F_1}\left(\begin{matrix}-m,\frac{\alpha+1}{2}\\{\alpha} \end{matrix};1-x\right) 
=\frac{\left(\frac{\alpha -1}{2}\right)_{m}}{m!}{_{2}F_1}\left(\begin{matrix}-m,\frac{\alpha+1}{2}\\{-m-\frac{\alpha -3}{2}} \end{matrix};x\right) \nonumber \\
&=\frac{(\alpha)_{m}}{m!}\sum_{k=0}^{m}(-1)^{k}\binom{m}{k}\frac{\left(\frac{\alpha+1}{2}\right)_{k}}{(\alpha)_{k}}(1-x)^{k}.
\end{align}
Here, $(\alpha)_{m}:=\frac{\Gamma(\alpha+m)}{\Gamma(\alpha)}=\alpha(\alpha+1)\cdots(\alpha+m-1)$ and $\binom{m}{k}:=(-1)^{k}\frac{(-m)_{k}}{k!}$. 
For $\alpha >0$, these polynomials $\phi_{m}^{(\alpha)}$ satisfy with the following orthogonality which were given by R.\,Askey \cite{A} in 1982 (he prove more general results).  
\begin{align}
\label{eq:circular Jacobi orthogonality}
\frac{1}{2\pi}\int_{0}^{2\pi}\!\!\!\!\phi_{m}^{(\alpha)}(e^{i\theta})\overline{\phi_{n}^{(\alpha)}(e^{i\theta})}
|(1-e^{i\theta})^{\frac{\alpha-1}{2}}|^{2}\,d\theta
=\frac{\Gamma(\alpha+m)}{m!}\frac{1}{\Gamma(\frac{\alpha+1}{2})^{2}}\delta_{mn}.
\end{align}
Li-Chien Shen \cite{She} provided a useful framework for introducing the circular Jacobi polynomials. 
Let us describe Shen's picture 
more precisely. 
We put $\mathcal{D}:=\{w \in \mathbb{C}\mid \vert w \vert<1\}$, $H:=\{z \in \mathbb{C}\mid \Im{z}>0\}$, $\Sigma :=\{\sigma \in \mathbb{C} \mid \sigma^{-1}=\overline{\sigma} \}$, $\mu$ is the measure associated with the Riemannian structure on $\Sigma$. 
Further, we consider the following function spaces and their complete orthogonal bases. 
\\
{\bf{(1)}}\,\,$\psi_{m}^{(\alpha)}$\,;\,exponential multiplied by Laguerre polynomials 
\begin{align}
L^{2}_{\alpha}(\mathbb{R}_{>0})&:=\{\psi:\mathbb{R}_{>0} \longrightarrow \mathbb{C} \mid \|\psi\|_{\alpha,\mathbb{R}_{>0}}^{2}<\infty\}, \nonumber \\
\|\psi\|_{\alpha,\mathbb{R}_{>0}}^{2}&:=\frac{2^{\alpha}}{\Gamma(\alpha)}\int_{0}^{\infty}|\psi(u)|^{2}u^{\alpha -1}\,du, \nonumber \\
\psi_{m}^{(\alpha)}(u)&:=e^{-u}L_{m}^{(\alpha-1)}(2u)
=\frac{(\alpha)_{m}}{m!}e^{-u}\sum_{k=0}^{m}(-1)^{k}\binom{m}{k}\frac{1}{(\alpha)_{k}}(2u)^{k}.\nonumber
\end{align}
{\bf{(2)}}\,\,$\Psi_{m}^{(\alpha)}$\,;\,Modified Fourier transform of the Laguerre polynomials
\begin{align}
H^{2}_{\alpha}(\mathbb{R})&:=\left\{\Psi:\mathbb{R} \longrightarrow \mathbb{C} \mid \text{$\|\Psi\|_{\alpha,\mathbb{R}}^{2}<\infty$ and $\Psi$ is continued analytically to $H$} \right. \nonumber \\
{} & \quad \quad \quad \quad \quad \quad \quad \quad \quad \left. \text{as a holomorphic function which satisfies with} \right. \nonumber \\
{} & \quad \quad \quad \quad \quad \quad \quad \quad \quad \left. \sup_{0<y<\infty}\frac{1}{2\pi}\int_{0}^{\infty}|\Psi(x+iy)|^{2}\,dx<\infty \right\}, \nonumber \\
\|\Psi\|_{\alpha,\mathbb{R}}^{2}&:=\frac{\Gamma\left(\frac{\alpha+1}{2}\right)^{2}}{2\pi}\frac{2^{\alpha}}{\Gamma(\alpha)}\int_{-\infty}^{\infty}
|\Psi(t)|^{2}\,dt, \nonumber \\
\Psi_{m}^{(\alpha)}(t)&
:=(1-it)^{-\frac{\alpha +1}{2}}\frac{(\alpha)_{m}}{m!}\sum_{k=0}^{m}(-1)^{k}\binom{m}{k}\frac{\left(\frac{\alpha+1}{2}\right)_{k}}{(\alpha)_{k}}\left(\frac{2}{1-it}\right)^{k}. \nonumber
\end{align}
{\bf{(3)}}\,\,$\phi_{m}^{(\alpha)}$\,;\,circular Jacobi polynomials
\begin{align}
H^{2}_{\alpha}(\Sigma)&:=\{\phi:\Sigma \longrightarrow \mathbb{C} \mid \text{$\phi$ is continued analytically to $\mathcal{D}$ as a holomorphic function} \nonumber \\
{} & \quad \quad \quad \quad \quad \quad \quad \quad \quad  \text{and $\|\phi\|_{\alpha,\Sigma}^{2}<\infty$}\}, \nonumber \\
\|\phi\|_{\alpha,\Sigma}^{2}&:=\frac{\Gamma\left(\frac{\alpha+1}{2}\right)^{2}}{2\pi{i}}\frac{1}{\Gamma(\alpha)}
\int_{\Sigma}|\phi(\sigma)|^{2}|(1-\sigma)^{\frac{\alpha-1}{2}}|^{2}\,d\mu(\sigma), \nonumber \\
\phi_{m}^{(\alpha)}(\sigma)
&:=\frac{(\alpha)_{m}}{m!}\sum_{k=0}^{m}(-1)^{k}\binom{m}{k}\frac{\left(\frac{\alpha+1}{2}\right)_{k}}{(\alpha)_{k}}(1-\sigma)^{k}.\nonumber
\end{align}
We remark that 
\begin{equation}
\|\psi_{m}^{(\alpha)}\|_{\alpha,\mathbb{R}_{>0}}^{2}=\|\Psi_{m}^{(\alpha)}\|_{\alpha,\mathbb{R}}^{2}=\|\phi_{m}^{(\alpha)}\|_{\alpha,\Sigma}^{2}=\frac{(\alpha)_{m}}{m!}. \nonumber
\end{equation}
Furthermore, the following unitary isomorphisms are known. \\
\underline{Modified (inverse) Fourier transform}
\begin{align}
\mathcal{F}_{\alpha}^{-1}:L^{2}_{\alpha}(\mathbb{R}_{>0}) \xrightarrow{\simeq}  H^{2}_{\alpha}(\mathbb{R}),
\,\,\,(\mathcal{F}_{\alpha}^{-1}\psi)(t):=\frac{1}{\Gamma\left(\frac{\alpha+1}{2}\right)}\int_{0}^{\infty}e^{itu}u^{\frac{\alpha-1}{2}}\psi(u)\,du.\nonumber
\end{align}
\underline{Modified Cayley transform}
\begin{align}
\mathcal{C}_{\alpha}^{-1}:H^{2}_{\alpha}(\mathbb{R}) \xrightarrow{\simeq} H^{2}_{\alpha}(\Sigma),
\,\,\,(\mathcal{C}_{\alpha}^{-1}\Psi)(\sigma):=\left(\frac{1-\sigma}{2}\right)^{-\frac{\alpha+1}{2}}\Psi\left(i\frac{1+\sigma}{1-\sigma}\right). \nonumber
\end{align}

To summarize, we can describe the result of Shen. 
\begin{prop}[\cite{She}] 
\label{prop:She}
\begin{align}
\begin{array}{cccccc}
L^{2}_{\alpha}(\mathbb{R}_{\geq 0}) & \xrightarrow[\mathcal{F}_{\alpha}^{-1}]{\simeq} & H^{2}_{\alpha}(\mathbb{R}) & \xrightarrow[\mathcal{C}_{\alpha}^{-1}]{\simeq} & H^{2}_{\alpha}(\Sigma),& \,\,\,(\text{\rm{unitary}}). \nonumber \\
\rotatebox{90}{$\in$}  & & \rotatebox{90}{$\in$} & & \rotatebox{90}{$\in$} & \\ 
\psi_{m}^{(\alpha)} & \longmapsto  & \Psi_{m}^{(\alpha)} & \longmapsto  & \phi_{m}^{(\alpha)} & \\
{\bf{(1)}} &  & {\bf{(2)}} &  & {\bf{(3)}} &
\end{array}
\end{align}
\end{prop}
This setting is not only beneficial for introducing the above orthogonal systems, but also studying their fundamental properties (orthogonality, generating functions, differential equations). 

The purpose of this article is to provide a multivariate analogue of the results obtained by Shen. 
Namely, we consider a modified Fourier transform of $L^{2}_{\alpha}(\Omega)^{K}$ and multivariate Laguerre polynomials. 
Using this unitary isomorphism and the modified Cayley transform, we introduce some new multivariate special orthogonal polynomials, which are a multivariate  analogue of circular Jacobi polynomials. 
These polynomials, which we call multivariate circular Jacobi (MCJ) polynomials, are generalizations of the spherical (zonal) polynomials that are different from the Jack or Macdonald polynomials, which are well known as an extension of spherical polynomials. 
We also remark that the weight function of their orthogonality relation coincides with the circular Jacobi ensemble defined by Bourgade\,et\,al. \cite{BNR}. 
Furthermore, we provide a generating function for the MCJ polynomials and a differential equation that is satisfied by the modified Cayley transform of the MCJ polynomials. In case of the multiplicity $d=2$, we establish a determinant formula for the MCJ polynomials.  

Let us now describe the content of the following sections. 
The basic definitions and fundamental properties of Jordan algebras and symmetric cones, 
and lemmas for analysis on symmetric cones have been presented in the first subsection of Section\,2, so that they can be referred to later. 
The next subsection presents a compilation of basic facts for the multivariate Laguerre polynomials. 

Based on these preparations, in Section\,3, 
we establish a generalization of Proposition\,\ref{prop:She}. 
In addition, using the generalization, we obtain the MCJ polynomials and their fundamental properties.  

Finally, in Section\,4, we present a conjecture and some problems for a further generalization of the MCJ polynomials.

\section{Preliminaries}
Throughout the paper, we denote the ring of rational integers by $\mathbb{Z}$, 
the field of real numbers by $\mathbb{R}$, the field of complex numbers by $\mathbb{C}$, 
the partition set of length $r$ by $\mathscr{P}$ 
\begin{equation}
\mathscr{P}:=\{\mathbf{m}=(m_{1}, \ldots, m_{r}) \in \mathbb{Z}_{\geq 0}^{r}\mid m_{1}\geq \cdots \geq m_{r}\}.
\end{equation}
For any vector $\mathbf{s}=(s_{1},\ldots,s_{r})\in \mathbb{C}^{r}$, we put 
\begin{align}
\Re{\mathbf{s}}&:=(\Re{s_{1}},\ldots,\Re{s_{r}}), \\
|\mathbf{s}|&:=s_{1}+\cdots+s_{r}, \\
\|\mathbf{s}\|&:=(|s_{1}|,\ldots,|s_{r}|).
\end{align}
Moreover, for $\mathbf{m} \in \mathscr{P}$ 
$$
\mathbf{m}!:=m_{1}!\cdots m_{r}!
$$
and we set $\delta:=(r-1,r-2,\ldots,1,0)$.
Refer to Faraut and Koranyi \cite{FK} for the details in this chapter.

\subsection{Analysis on symmetric cones}
Let $\Omega$ be an irreducible symmetric cone in $V$ which is a finite dimensional simple Euclidean Jordan algebra of dimension $n$ as a real vector space and rank $r$. 
The classification of irreducible symmetric cones is well-known. 
Namely, there are four families of classical irreducible symmetric cones $\Pi_{r}(\mathbb{R}), \Pi_{r}(\mathbb{C}), \Pi_{r}(\mathbb{H})$, 
the cones of all $r\times r$ positive definite matrices over $\mathbb{R}$, $\mathbb{C}$ and $\mathbb{H}$, the Lorentz cones $\Lambda_{r}$ and an exceptional cone $\Pi_{3}(\mathbb{O})$ (see \cite{FK}\,p.\,97). 
Also, let $V^{\mathbb{C}}:=V+iV$ be its complexification and $T_{\Omega}:=\Omega +iV$, and $H_{\Omega}:=V+i\Omega$. 
We denote the Jordan trace and determinant of the complex Jordan algebra $V^{\mathbb{C}}$ by $\tr{x}$ and by $\Delta(x)$ respectively.

Fix a Jordan frame $\{c_{1},\ldots,c_{r}\}$ that is a complete system of orthogonal primitive idempotents in $V$ and define the following subspaces:
\begin{align}
V_{j}&:=\{x \in V \mid L(c_{j})x=x\}, \nonumber \\
V_{jk}&:=\left\{x \in V \bigg| L(c_{j})x=\frac{1}{2}x \,\,\text{and}\,\, L(c_{k})x=\frac{1}{2}x\right\}. \nonumber
\end{align}
Then, $V_{j}=\mathbb{R}e_{j}$ for $j=1,\ldots,r$ are $1$-dimensional subalgebras of $V$, 
while the subspaces $V_{jk}$ for $j,k=1,\ldots,r$ with $j<k$ all have a common dimension $d=\dim_{\mathbb{R}}V_{jk}$. 
Then, $V$ has the Peirce decomposition 
\begin{equation}
V=\left(\bigoplus_{j=1}^{r}{V_{j}}\right)\oplus\left(\bigoplus_{j<k}{V_{jk}}\right), \nonumber 
\end{equation}
which is the orthogonal direct sum. 
It follows that $n=r+\frac{d}{2}r(r-1)$. 
Let $G(\Omega)$ denote the automorphism group of $\Omega$ and let $G$ be the identity component in $G(\Omega)$. 
Then, $G$ acts transitively on $\Omega$ and $\Omega \cong G/K$ where $K \in G$ is the isotropy subgroup of the unit element, $e \in V$. 
$K$ is also the identity component in $\Aut(V)$.

For any $x \in V$, there exists $k \in K$ and $\lambda_{1},\ldots, \lambda_{r} \in \mathbb{R}$ such that 
\begin{equation}
x={k}\sum_{j=1}^{r}{\lambda_{j}c_{j}},\,\,\,\,(\lambda_{1}\geq \cdots \geq \lambda_{r}). \nonumber 
\end{equation}
From this polar decomposition, we obtain the following formula (see \cite{FK}\,Theorem\,VI.\,2.3).
\begin{lem}
\label{thm:int poral decomp}
Let $f$ be an integrable function on V. We have
\begin{equation}
\int_{V}f(x)\,dx=\widetilde{c_{0}}\int_{K\times \mathbb{R}^{r}}f(k\lambda)\prod_{1\leq p<q\leq r}|\lambda_{p}-\lambda_{q}|^{d}\,{dk}{d}\lambda_{1}\cdots{d}\lambda_{r}.
\end{equation}
Here, $dx$ is the Euclidean measure associated with the Euclidean structure on $V$ given by $(u|v)={\rm{tr}}(uv)$,  
$dk$ is the normalized Haar measure on the compact group $K$, $\lambda=\sum_{j=1}^{r}\lambda_{j}c_{j}$ and $\widetilde{c_{0}}$ is defined by 
\begin{equation}
\label{eq:def of c_{0}}
\widetilde{c_{0}}:=(2\pi)^{\frac{n-r}{2}}\prod_{j=1}^{r}\frac{\Gamma\left(\frac{d}{2}+1\right)}{\Gamma\left(\frac{d}{2}j+1\right)}
=\frac{(2\pi)^{\frac{n-r}{2}}}{r!}\prod_{j=1}^{r}\frac{\Gamma\left(\frac{d}{2}\right)}{\Gamma\left(\frac{d}{2}j\right)}.
\end{equation}
In particular, for $f \in L^{1}(V)^{K}$ 
\begin{equation}
\label{eq:int poral decomp}
\int_{V}f(x)\,dx=\widetilde{c_{0}}\int_{\mathbb{R}^{r}}f(\lambda_{1},\ldots,\lambda_{r})\prod_{1\leq p<q\leq r}|\lambda_{p}-\lambda_{q}|^{d}\,{d}\lambda_{1}\cdots{d}\lambda_{r}.
\end{equation}
\end{lem}
As in the case of $V$, we also have the following spectral decomposition for $V^{\mathbb{C}}$. 
Every $z$ in $V^{\mathbb{C}}$ can be written
$$
z=u\sum_{j=1}^{r}\lambda_{j}c_{j},
$$
with $u$ in $U$ which is the identity component of $Str(V^{\mathbb{C}})\cap{U(V^{\mathbb{C}})}$, $\lambda_{1}\geq \cdots\geq \lambda_{r}\geq 0$. 
Moreover, we define the spectral norm of $z \in V^{\mathbb{C}}$ by $|z|=\lambda_{1}$ and introduce open unit ball $\mathcal{D} \in V^{\mathbb{C}}$ as follows. 
\begin{equation}
\mathcal{D}=\{z \in V^{\mathbb{C}} \mid |z|<1\}. \nonumber
\end{equation}

We define $\Sigma$ as the set of invertible elements in $V^{\mathbb{C}}$ such that $z^{-1}=\overline{z}$, which coincides with the Shilov boundary of $\mathcal{D}$. 
For $\Sigma$, the following result is well known (see \cite{FK} Proposition\,X.2.3). 
\begin{lem}
\label{thm:fundamental properties for Shilov boundary}
For $z \in V^{\mathbb{C}}$, the following properties are equivalent:

\noindent
{\rm{(i)}}\,$z \in \Sigma$,

\noindent
{\rm{(i\hspace{-.1em}i)}}\,$z=e^{i\theta}=\sum_{j=1}^{r}e^{i\theta_{j}}c_{j}$ with $\theta=\sum_{j=1}^{r}\theta_{j}c_{j} \in V$,

\noindent
{\rm{(i\hspace{-.1em}i\hspace{-.1em}i)}}\,$z \in \overline{c^{-1}(V)}$,

\noindent
where $c^{-1}(t):=(t-ie)(t+ie)^{-1}=e-2i(t+ie)^{-1}$ is called the inverse Cayley transform. 
\end{lem}
We will later need the following integral formula on $\Sigma$ to describe the MCJ polynomials. 
\begin{lem}
\label{thm:int for Shilov boundary}

Let $\mu$ denote the measure associated with the Riemannian structure on $\Sigma$ induced by the Euclidean structure of $V^{\mathbb{C}}$. 

\noindent
{\rm{(1)}}\,If $\phi$ is an integrable function on $\Sigma$, then 
\begin{align}
\label{eq:int for Shilov boundary1}
\int_{\Sigma}\phi(\sigma)\,d\mu(\sigma)=2^{n}\int_{V}\phi(c^{-1}(t))|\Delta(e-it)^{-\frac{n}{r}}|^{2}\,dt.
\end{align}

\noindent
{\rm{(2)}}\,If $\Psi$ is an integrable function on $V$, then 
\begin{equation}
\label{eq:int for Shilov boundary2}
\int_{V}\Psi(t)\,dt=2^{n}\int_{\Sigma}\Psi(c(\sigma))|\Delta(e-\sigma)^{-\frac{n}{r}}|^{2}\,d\mu(\sigma). 
\end{equation}
Here, $c$ is a Cayley transform defined by $c(\sigma):=i(e+\sigma)(e-\sigma)^{-1}=-ie+2i(e-\sigma)^{-1}$.

\noindent
{\rm{(3)}}\,If $\Psi$ is an integrable function on $V$ and a $K$-invariant, then
\begin{equation}
\label{eq:int for Shilov boundary3}
\int_{\Sigma}\phi(\sigma)\,d\mu(\sigma)=\widetilde{c_{0}}\int_{\mathcal{S}^{r}}\phi(e^{i\theta})\prod_{1\leq p<q\leq r}|e^{i\theta_{p}}-e^{i\theta_{q}}|^{d}\,d\theta_{1}\cdots d\theta_{r}.  
\end{equation}
Here, $\mathcal{S}^{r}$ is the direct product of $r$ copies of $S^{1}$. 
\end{lem}
\begin{proof}
{\rm{(1)}} is Proposition\,X.2.4 of \cite{FK} itself and {\rm{(2)}} also immediately follows from some proposition. 
Hence, we only prove {\rm{(3)}}.

Let $\phi \in L^{1}(\Sigma)^{K}$. 
Since for any $k \in K$
$$
c^{-1}(kt)=(k(t-ie))(k(t+ie))^{-1}
=k((t+ie)(t-ie)^{-1})
=kc^{-1}(t),
$$
from Lemma\,\ref{thm:int poral decomp}, we have
\begin{align}
\int_{\Sigma}\phi(\sigma)\,d\mu(\sigma)
&=2^{n}\int_{V}\phi(c^{-1}(t))\Delta(e+t^{2})^{-\frac{n}{r}}\,dt \nonumber \\
&=2^{n}\widetilde{c_{0}}\int_{K\times \mathbb{R}^{r}}\phi(c^{-1}(k\lambda ))\Delta(e+(k\lambda)^{2})^{-\frac{n}{r}}\prod_{1\leq p<q\leq r}|\lambda_{p}-\lambda_{q}|^{d}\,{dk}{d}\lambda_{1}\cdots{d}\lambda_{r} \nonumber \\
&=2^{n}\widetilde{c_{0}}\int_{\mathbb{R}^{r}}\phi(c^{-1}(\lambda ))\Delta(e+\lambda^{2})^{-\frac{n}{r}}\prod_{1\leq p<q\leq r}|\lambda_{p}-\lambda_{q}|^{d}\,{d}\lambda_{1}\cdots{d}\lambda_{r}. \nonumber
\end{align}
If we put $\lambda_{j}=-\cot\left(\frac{\theta_{j}}{2}\right)$, then
$$
\lambda =-\sum_{j=1}^{r}\cot\left(\frac{\theta_{j}}{2}\right)c_{j}
=i\sum_{j=1}^{r}\frac{1+e^{i\theta_{j}}}{1-e^{i\theta_{j}}}c_{j}
=i\left(\sum_{j=1}^{r}(1+e^{i\theta_{j}})c_{j}\right)\left(\sum_{l=1}^{r}(1-e^{i\theta_{l}})c_{l}\right)^{-1}
=c(e^{i\theta}).
$$
Therefore, 
\begin{align}
\int_{\Sigma}\phi(\sigma)\,d\mu(\sigma)
&=2^{n-r}\widetilde{c_{0}}\int_{\mathcal{S}^{r}}\phi(e^{i\theta})\prod_{j=1}^{r}\sin\left(\frac{\theta_{j}}{2}\right)^{2\left(\frac{n}{r}-1\right)}\!\!\!\!\!\!\!
\prod_{1\leq p<q\leq r}\left|\frac{\sin\left(\frac{1}{2}(\theta_{p}-\theta_{q})\right)}{\sin\left(\frac{\theta_{p}}{2}\right)\sin\left(\frac{\theta_{q}}{2}\right)}\right|^{d}
\,d\theta_{1}\cdots{d}\theta_{r} \nonumber \\
&=\widetilde{c_{0}}\int_{\mathcal{S}^{r}}\phi(e^{i\theta})\prod_{1\leq p<q\leq r}|e^{i\theta_{p}}-e^{i\theta_{q}}|^{d}\,d\theta_{1}\cdots d\theta_{r}. \nonumber
\end{align}
\end{proof}

For $j=1,\ldots,r$, let $e_{j}:=c_{1}+\cdots+c_{j}$, and set 
\begin{equation}
V^{(j)}:=\{x \in V \mid L(e_{j})x=x\}. \nonumber
\end{equation}
Denote the orthogonal projection of $V$ onto the subalgebra $V^{(j)}$ by $P_{j}$, and define 
\begin{equation}
\Delta_{j}(x):=\delta_{j}(P_{j}x) \nonumber   
\end{equation}
for $x \in V$, where $\delta_{j}$ denotes the Koecher norm function for $V^{(j)}$. 
In particular, $\delta_{r}=\Delta$. 
Then, $\Delta_{j}$ is a polynomial on $V$ that is homogeneous of degree $j$. 
Let $\mathbf{s}:=(s_{1},\ldots,s_{r}) \in \mathbb{C}^{r}$ and define the function $\Delta_{\mathbf{s}}$ on $V$ by 
\begin{equation}
\Delta_{\mathbf{s}}(x):=\Delta(x)^{s_{r}}\prod_{j=1}^{r-1}\Delta_{j}(x)^{s_{j}-s_{j+1}}. 
\end{equation}
Here, we define the branch by $\Delta_{\mathbf{s}}(e)=1$. 
That is the generalized power function on $V$. 
In particular, for $\mathbf{m} \in \mathscr{P}$, $\Delta_{\mathbf{m}}$ becomes a polynomial function on $V$, which is homogeneous of degree $|\mathbf{m}|$. 
Furthermore, $\Delta_{\mathbf{s}}$ can be extended to the function on $V^{\mathbb{C}}$ by analytic continuation. 

The gamma function $\Gamma_{\Omega}$ for the symmetric cone $\Omega$ is defined, for $\mathbf{s} \in \mathbb{C}^{r}$, with $\Re{s_{j}}>\frac{d}{2}(j-1)\,(j=1,\ldots,r)$ by 
\begin{equation}
\Gamma_{\Omega}(\mathbf{s}):=\int_{\Omega}e^{-{\rm{tr}}(x)}\Delta_{\mathbf{s}}(x)\Delta(x)^{-\frac{n}{r}}\,dx.
\end{equation}
Its evaluation gives 
\begin{equation}
\label{eq:def of gamma on cone}
\Gamma_{\Omega}(\mathbf{s})=(2\pi)^{\frac{n-r}{2}}\prod_{j=1}^{r}\Gamma\left(s_{j}-\frac{d}{2}(j-1)\right).
\end{equation}
Hence, $\Gamma_{\Omega}$ extends analytically as a meromorphic function on $\mathbb{C}^{r}$. 


For $\mathbf{s} \in \mathbb{C}^{r}$ and $\mathbf{m} \in \mathscr{P}$, we define the generalized shifted factorial by 
\begin{equation}
(\mathbf{s})_{\mathbf{m}}:=\frac{\Gamma_{\Omega}(\mathbf{s}+\mathbf{m})}{\Gamma_{\Omega}(\mathbf{s})}.
\end{equation}
It follows from (\ref{eq:def of gamma on cone}) that 
\begin{equation}
(\mathbf{s})_{\mathbf{m}}=\prod_{j=1}^{r}\left(s_{j}-\frac{d}{2}(j-1)\right)_{m_{j}}.
\end{equation}
\begin{lem}
\label{thm:ineq for generalized shifted factorial}
If $\mathbf{s} \in \mathbb{C}^{r}, \mathbf{m},\mathbf{k} \in \mathscr{P}$ and $\mathbf{m}\supset \mathbf{k}$, then
\begin{equation}
\label{eq:ineq for generalized shifted factorial}
\left|\frac{(\mathbf{s})_{\mathbf{m}}}{(\mathbf{s})_{\mathbf{k}}}\right| \leq \frac{(\|\mathbf{s}\|+d(r-1))_{\mathbf{m}}}{(\|\mathbf{s}\|+d(r-1))_{\mathbf{k}}}.
\end{equation}
\end{lem}
\begin{proof}
We remark that for any $s \in \mathbb{C}, N \in \mathbb{Z}_{\geq 0}$ and $j=1,\ldots,r$, the following is satisfied.
\begin{equation}
\label{eq:ineq for generalized shifted factorial prot}
\left|s+N-\frac{d}{2}(j-1)\right|\leq |s|+N+d(r-1)-\frac{d}{2}(j-1)
=|s|+N+\frac{d}{2}(2r-j-1). \nonumber
\end{equation}
Hence, 
\begin{align}
\left|\frac{(\mathbf{s})_{\mathbf{m}}}{(\mathbf{s})_{\mathbf{k}}}\right|
&=\prod_{j=1}^{r}\left|\left(s_{j}+k_{j}-\frac{d}{2}(j-1)\right)_{m_{j}-k_{j}}\right| \nonumber \\
&\leq \prod_{j=1}^{r}\left(|s_{j}|+k_{j}+d(r-1)-\frac{d}{2}(j-1)\right)_{m_{j}-k_{j}} \nonumber \\
&=\frac{\left(\|\mathbf{s}\|+d(r-1)\right)_{\mathbf{m}}}{\left(\|\mathbf{s}\|+d(r-1)\right)_{\mathbf{k}}}. \nonumber 
\end{align}
\end{proof}
\begin{cor}
\label{thm:ineq for generalized shifted factorial 2}
If $\mathbf{s} \in \mathbb{C}^{r}, \mathbf{m} \in \mathscr{P}$, then
\begin{equation}
\label{eq:ineq for generalized shifted factorial2}
|(\mathbf{s})_{\mathbf{m}}| \leq (\|\mathbf{s}\|+d(r-1))_{\mathbf{m}} \leq \prod_{j=1}^{r}(|s_{j}|+d(r-1))_{m_{j}}.
\end{equation}
\end{cor}

The space, $\mathcal{P}(V)$, of the polynomial ring on $V$ has the following decomposition. 
\begin{equation}
\mathcal{P}(V)=\bigoplus_{\mathbf{m} \in \mathscr{P}}\mathcal{P}_{\mathbf{m}}, \nonumber 
\end{equation}
where each $\mathcal{P}_{\mathbf{m}}$ are mutually inequivalent, and finite dimensional irreducible $G$-modules. 
Further, their dimensions are denoted by $d_{\mathbf{m}}$. 
For $d_{\mathbf{m}}$, the following formula is known (see, \cite{Up}\,Lemma 2.6 or \cite{FK}\,p.\,315). 
\begin{lem}
\label{thm:Upmeier Lem 2.6}
For any $\mathbf{m} \in \mathscr{P}$,
\begin{align}
\label{eq:Upmeier Lem 2.6}
d_{\mathbf{m}}
&=\frac{c(-\rho)}{c(\rho -\mathbf{m})c(\mathbf{m}-\rho)} \\
&=\prod_{1\leq p<q\leq r}\frac{m_{p}-m_{q}+\frac{d}{2}(q-p)}{\frac{d}{2}(q-p)}
\frac{B\left(m_{p}-m_{q},\frac{d}{2}(q-p-1)+1\right)}{B\left(m_{p}-m_{q},\frac{d}{2}(q-p+1)\right)} \\
&=\prod_{j=1}^{r}\frac{\Gamma\left(\frac{d}{2}\right)}{\Gamma\left(\frac{d}{2}j\right)\Gamma\left(\frac{d}{2}(j-1)+1\right)} \nonumber \\
{} & \quad \cdot
\prod_{1\leq p<q\leq r}\left(m_{p}-m_{q}+\frac{d}{2}(q-p)\right)\frac{\Gamma\left( m_{p}-m_{q}+\frac{d}{2}(q-p+1)\right)}{\Gamma\left(m_{p}-m_{q}+\frac{d}{2}(q-p-1)+1\right)}.
\end{align}
Here, $\rho=(\rho_{1},\ldots,\rho_{r})$, $\rho_{j}:=\frac{d}{4}(2j-r-1)$, and $c$ is the Harish-Chandra function:
$$
c(\mathbf{s})=\prod_{1\leq p<q\leq r}\frac{B\left(s_{q}-s_{p}, \frac{d}{2}\right)}{B\left(\frac{d}{2}(q-p), \frac{d}{2}\right)}.
$$
In particular, for $d=2$ 
\begin{equation}
\label{eq:d and Schur}
d_{\mathbf{m}}=\prod_{1\leq p<q\leq r}\left(\frac{m_{p}-m_{q}+q-p}{q-p}\right)^{2}=s_{\mathbf{m}}(1,\ldots,1)^{2}.
\end{equation}
Here, $s_{\mathbf{m}}$ is the Schur polynomial corresponding to $\mathbf{m} \in \mathscr{P}$ defined by 
$$
s_{\mathbf{m}}(\lambda_{1},\ldots,\lambda_{r}):=\frac{\det(\lambda_{j}^{m_{k}+r-k})}{\det(\lambda_{j}^{r-k})}.
$$
\end{lem}
The following lemma is necessary to evaluate the Fourier transform of the multivariate Laguerre polynomial.
\begin{lem}[$\cite{FK}$\,Theorem\,XI.\,$2.3$]
\label{thm:int formula}
For $p \in \mathcal{P}_{\mathbf{m}}$, $\Re{\alpha}>(r-1)\frac{d}{2}$, and $z \in T_{\Omega}$,
\begin{equation}
\int_{\Omega}e^{-(z|x)}p(x)\Delta(x)^{\alpha -\frac{n}{r}}\,dx=\Gamma_{\Omega}(\mathbf{m}+\alpha)\Delta(z)^{-\alpha}p(z^{-1}).
\end{equation}
Here, $\alpha$ is regarded as $(\alpha,\ldots,\alpha) \in \mathbb{C}^{r}$. 
\end{lem}
For each $\mathbf{m} \in \mathscr{P}$, the spherical polynomial of weight $|\mathbf{m}|$ on $\Omega$ is defined by 
\begin{equation}
\Phi_{\mathbf{m}}^{(d)}(x):=\int_{K}\Delta_{\mathbf{m}}(kx)\,dk.
\end{equation}
We often omit multiplicity $d$ of $\Phi_{\mathbf{m}}^{(d)}(x)$. 
The algebra of all $K$-invariant polynomials on $V$, denoted by $\mathcal{P}(V)^{K}$, decomposes as 
\begin{equation}
\mathcal{P}(V)^{K}=\bigoplus_{\mathbf{m} \in \mathscr{P}}\mathbb{C}\Phi_{\mathbf{m}}. \nonumber 
\end{equation}
By analytic continuation to the complexification $V^{\mathbb{C}}$ of $V$, 
we can extend $\tr, \Delta$ and $\Phi_{\mathbf{m}}$ to polynomial functions on $V^{\mathbb{C}}$. 
\begin{rmk}
{\rm(1)}\,Since $\Phi_{\mathbf{m}}\in\mathcal{P}_{\mathbf{m}}^{K}$, for $x=k\sum_{j=1}^{r}\lambda_{j}c_{j}$, $\Phi_{\mathbf{m}}(x)$ can be expressed by
$$
\Phi_{\mathbf{m}}(\lambda_{1},\ldots,\lambda_{r}):=\Phi_{\mathbf{m}}\left(\sum_{j=1}^{r}\lambda_{j}c_{j}\right)(=\Phi_{\mathbf{m}}(x)).
$$
$\Phi_{\mathbf{m}}(x)$ also has the following expression (see \cite{F}). 
\begin{equation}
\Phi_{\mathbf{k}}^{(d)}(\lambda_{1},\ldots,\lambda_{r})=\frac{P_{\mathbf{k}}^{(\frac{2}{d})}(\lambda_{1},\ldots,\lambda_{r})}{P_{\mathbf{k}}^{(\frac{2}{d})}(1,\ldots,1)}.
\end{equation}
Here, $P_{\mathbf{k}}^{(\frac{2}{d})}(\lambda_{1},\ldots,\lambda_{r})$ is an $r$-variable Jack polynomial (see \cite{M}, Chapter. VI.10). 
In particular, since $P_{\mathbf{k}}^{(1)}(\lambda_{1},\ldots,\lambda_{r})=s_{\mathbf{m}}(\lambda_{1},\ldots,\lambda_{r})$, $\Phi_{\mathbf{m}}^{(2)}$ becomes the Schur polynomial. 
\begin{equation}
\label{eq:spherical and schur}
\Phi_{\mathbf{m}}^{(2)}(\lambda_{1},\ldots,\lambda_{r})=\frac{s_{\mathbf{m}}(\lambda_{1},\ldots,\lambda_{r})}{s_{\mathbf{m}}(1,\ldots,1)}=\frac{\delta!}{\prod_{p<q}(m_{p}-m_{q}+q-p)}s_{\mathbf{m}}(\lambda_{1},\ldots,\lambda_{r}).
\end{equation}
\noindent
{\rm(2)}\,When $r=2$, $\Phi_{\mathbf{m}}^{(d)}$ has the following hypergeometric expression\,(see \cite{Sa}). 
\begin{align}
\Phi_{m_{1},m_{2}}^{(d)}(\lambda_{1},\lambda_{2})
&=\lambda_{1}^{m_{1}}\lambda_{2}^{m_{2}}{_{2}F_1}\left(\begin{matrix}-(m_{1}-m_{2}),\frac{d}{2}\\d \end{matrix};\frac{\lambda_{1}-\lambda_{2}}{\lambda_{1}}\right) \nonumber \\
&=\lambda_{1}^{m_{1}}\lambda_{2}^{m_{2}}\frac{\left(\frac{d}{2}\right)_{m_{1}-m_{2}}}{(d)_{m_{1}-m_{2}}}{_{2}F_1}\left(\begin{matrix}-(m_{1}-m_{2}),\frac{d}{2}\\-(m_{1}-m_{2})-\frac{d}{2}+1 \end{matrix};\frac{\lambda_{2}}{\lambda_{1}}\right). \nonumber 
\end{align}
\end{rmk}
We remark that the function $\Phi_{\mathbf{m}}(e+x)$ is a $K$-invariant polynomial of degree $|\mathbf{m}|$ and define the generalized binomial coefficients $\binom{\mathbf{m}}{\mathbf{k}}_{\frac{d}{2}}$ by using the following expansion.
\begin{equation}
\Phi_{\mathbf{m}}^{(d)}(e+x)=\sum_{|\mathbf{k}|\leq |\mathbf{m}|}\binom{\mathbf{m}}{\mathbf{k}}_{\frac{d}{2}}\Phi_{\mathbf{k}}^{(d)}(x).
\end{equation}
For $\binom{\mathbf{m}}{\mathbf{k}}_{\frac{d}{2}}$, we also often omit $\frac{d}{2}$. 
The fact that if $\mathbf{k} \not\subset \mathbf{m}$, then $\binom{\mathbf{m}}{\mathbf{k}}=0$, is well known. 
Hence, we have 
\begin{equation}
\label{eq:the definition of the generalized binomial coefficients}
\Phi_{\mathbf{m}}(e+x)=\sum_{\mathbf{k} \subset \mathbf{m}}\binom{\mathbf{m}}{\mathbf{k}}\Phi_{\mathbf{k}}(x).
\end{equation}

%

\begin{lem}
\label{thm:FK,Thm12.1.1}
For $z=u\sum_{j=1}^{r}\lambda_{j}c_{j}$ with $u \in U$, $\lambda_{1}\geq \cdots\geq \lambda_{r}\geq 0$ and $\mathbf{m} \in \mathscr{P}$, we have
\begin{equation}
\label{eq:FK,Thm12.1.1}
|\Phi_{\mathbf{m}}(z)|\leq \lambda_{1}^{m_{1}}\cdots \lambda_{r}^{m_{r}}\leq \lambda_{1}^{|\mathbf{m}|}=\Phi_{\mathbf{m}}(\lambda_{1}).
\end{equation}
\end{lem}

\begin{lem}
\label{thm:Cauchy kernel of spherical poly}
For any $\alpha \in \mathbb{C},z \in \overline{\mathcal{D}},w \in \mathcal{D}$, we have
\begin{equation}
\label{eq:Cauchy kernel of spherical poly}
\sum_{\mathbf{m} \in \mathscr{P}}d_{\mathbf{m}}\frac{(\alpha)_{\mathbf{m}}}{\left(\frac{n}{r}\right)_{\mathbf{m}}}\Phi_{\mathbf{m}}(z)\Phi_{\mathbf{m}}(w)
=\Delta(w)^{-\alpha}\int_{K}\Delta(kw^{-1}-z)^{-\alpha}\,dk.
\end{equation}
\end{lem}

The spherical function, $\varphi_{\mathbf{s}}$, on $\Omega$ for $\mathbf{s} \in \mathbb{C}^{r}$ is defined by 
\begin{equation}
\varphi_{\mathbf{s}}(x):=\int_{K}\Delta_{\mathbf{s}+\rho}(kx)\,dk.  
\end{equation}
We remark that for $x \in \Omega$
\begin{equation}
\varphi_{\mathbf{s}}(x^{-1})=\varphi_{-\mathbf{s}}(x)
\end{equation}
and  for $x \in \Omega, \mathbf{m} \in \mathscr{P}$
\begin{equation}
\Phi_{\mathbf{m}}(x)=\varphi_{\mathbf{m}-\rho}(x).  
\end{equation}

Let $\mathbb{D}(\Omega)$ be the algebra of $G$-invariant differential operators on $\Omega$, 
$\mathcal{P}(V)^{K}$ be the space of $K$-invariant polynomials on $V$, 
and $\mathcal{P}(V\times V)^{G}$ be the space of polynomials on $V\times V$, which are invariant in the sense that
$$
p(gx,\xi)=p(x,g^{*}\xi),\,\,\,\,(g \in G).
$$
Here, we write $g^{*}$ for the adjoint of an element $g$ (i.e., $(gx|y)=(x|g^{*}y)$ for all $x,y \in V$). 
The spherical function $\varphi_{\mathbf{s}}$ is an eigenfunction of every $D \in \mathbb{D}(\Omega)$. 
Thus, we denote its eigenvalues by $\gamma(D)(\mathbf{s})$, that is, $D\varphi_{\mathbf{s}}=\gamma(D)(\mathbf{s})\varphi_{\mathbf{s}}$. 

The symbol $\sigma_{D}$ of a partial differential operator $D$ which acts on the variable $x \in V$ is defined by 
\begin{equation}
De^{(x|\xi)}=\sigma_{D}(x,\xi)e^{(x|\xi)}\,\,\,\,(x,\xi \in V). \nonumber
\end{equation}
Differential operator $D$ on $\Omega$ is invariant under $G$ if and only if its symbol $\sigma_{D}$ belongs to $\mathcal{P}(V\times V)^{G}$. 
In addition, the map $D \mapsto \sigma_{D}$ establishes a linear isomorphism 
from $\mathbb{D}(\Omega)$ onto $\mathcal{P}(V\times V)^{G}$. 
Moreover, the map $D \mapsto \sigma_{D}(e,u)$ is a vector space isomorphism from $\mathbb{D}(\Omega)$ onto $\mathcal{P}(V)^{K}$. 
In particular, for $\mathbf{k} \in \mathscr{P}, \mathbf{s} \in \mathbb{C}^{r}$, we put
\begin{equation}
\gamma_{\mathbf{k}}(\mathbf{s}):=\gamma(\Phi_{\mathbf{k}}(\partial_{x}))(\mathbf{s})=\Phi_{\mathbf{k}}(\partial_{x})\varphi_{\mathbf{s}}(x)|_{x=e}.
\end{equation}
Here, $\Phi_{\mathbf{k}}(\partial_{x})$ is a unique $G$-invariant differential operator, which is satisfied with 
$$
\sigma_{\Phi_{\mathbf{k}}(\partial_{x})}(e,\xi)=\Phi_{\mathbf{k}}(\xi) \in \mathcal{P}(V)^{K},
\,\,\,\,\text{i.e.,}\,\,\Phi_{\mathbf{k}}(\partial_{x})e^{(x|\xi)}|_{x=e}=\Phi_{\mathbf{k}}(\xi)e^{\tr{\xi}}.
$$ 
We remark that $\Phi_{k}(\partial_{x})=\partial_{x}^{k}$ and $\gamma_{k}(s)=s(s-1)\cdots(s-k+1)$ in the $r=1$ case, and for any $\alpha \in \mathbb{C}$, $\mathbf{k} \in \mathscr{P}$, we have
\begin{equation}
\label{eq:gamma special case}
\gamma_{\mathbf{k}}(\alpha -\rho)=(-1)^{|\mathbf{k}|}(-\alpha)_{\mathbf{k}}.
\end{equation}   
The function $\gamma_{D}$ is an $r$ variable symmetric polynomial   
and map $D \mapsto \gamma_{D}$ is an algebra isomorphism from $\mathbb{D}(\Omega)$ onto algebra $\mathcal{P}(\mathbb{R}^{r})^{\mathfrak{S}_{r}}$, 
which is a special case of the Harish-Chandra isomorphism.

If a $K$-invariant function $\psi$ is analytic in the neighborhood of $e$, it admits a spherical Taylor expansion near $e$:
\begin{equation}
\psi(e+x)=\sum_{\mathbf{k} \in \mathscr{P}}d_{\mathbf{k}}\frac{1}{\left(\frac{n}{r}\right)_{\mathbf{k}}}\{\Phi_{\mathbf{k}}(\partial_{x})\psi(x)|_{x=e}\}\Phi_{\mathbf{k}}(x). \nonumber
\end{equation}
By the definition of $\gamma_{\mathbf{k}}$, we have
\begin{equation}
\varphi_{\mathbf{s}}(e+x)=\sum_{\mathbf{k} \in \mathscr{P}}d_{\mathbf{k}}\frac{1}{\left(\frac{n}{r}\right)_{\mathbf{k}}}\gamma_{\mathbf{k}}(\mathbf{s})\Phi_{\mathbf{k}}(x). \nonumber
\end{equation}
Since $\Phi_{\mathbf{m}}=\varphi_{\mathbf{m}-\rho}$, 
$$
\binom{\mathbf{m}}{\mathbf{k}}=d_{\mathbf{k}}\frac{1}{\left(\frac{n}{r}\right)_{\mathbf{k}}}\gamma_{\mathbf{k}}(\mathbf{m}-\rho).
$$

For a complex number $\alpha$, we define the following differential operator on $\Omega$:
\begin{equation}
D_{\alpha}=\Delta(x)^{1+\alpha}\Delta(\partial_{x})\Delta(x)^{-\alpha}. \nonumber
\end{equation}
For this operator, we have 
\begin{equation}
\gamma(D_{\alpha})(\mathbf{s})=\prod_{j=1}^{r}\left(s_{j}-\alpha+\frac{d}{4}(r-1)\right).
\end{equation}
The operators $D_{j\frac{d}{2}},\,j=0,\ldots,r-1$ generate algebra $\mathbb{D}(\Omega)$. 
\begin{lem}
\label{thm:estimate of shifted Jack1}
For all $\mathbf{k} \in \mathscr{P}$, there exist some constant $C>0$ and integer $N$ such that for any $\mathbf{s} \in \mathbb{C}^{r}$
\begin{align}
\label{eq:estimate of shifted Jack1}
|\gamma_{\mathbf{k}}(\mathbf{s})|&\leq C\prod_{l=1}^{r}\left(|s_{l}|+\frac{d}{4}(r-1)\right)^{N}.
\end{align}
\end{lem}
\begin{proof}
Since algebra $\mathbb{D}(\Omega)$ is generated by $D_{j\frac{d}{2}},\,j=0,\ldots,r-1$, 
for $\Phi_{\mathbf{k}}(\partial_{x}) \in \mathbb{D}(\Omega)$,  
$$
\Phi_{\mathbf{k}}(\partial_{x})=\sum_{l_{0},\ldots,l_{r-1};{\text{finite}}}
a_{l_{0},\ldots,l_{r-1}}D_{0\frac{d}{2}}^{l_{0}} \cdots D_{(r-1)\frac{d}{2}}^{l_{r-1}}.
$$
Here, we remark that for $j=0,\ldots,r-1$
\begin{align}
|\gamma(D_{\frac{d}{2}(j-1)})(\mathbf{s})|
=\left|\prod_{l=1}^{r}\left(s_{l}+\frac{d}{4}(r-1)-\frac{d}{2}(j-1)\right)\right| 
\leq \prod_{l=1}^{r}\left(|s_{l}|+\frac{d}{4}(r-1)\right). \nonumber
\end{align}
Therefore, 
\begin{align}
|\gamma_{\mathbf{k}}(\mathbf{s})|
\leq \sum_{l_{0},\ldots,l_{r-1};{\text{finite}}}
|a_{l_{0},\ldots,l_{r-1}}|\gamma(D_{0\frac{d}{2}})(\mathbf{s})^{l_{0}} \cdots \gamma(D_{(r-1)\frac{d}{2}})(\mathbf{s})^{l_{r-1}}
\leq C\prod_{l=1}^{r}\left(|s_{l}|+\frac{d}{4}(r-1)\right)^{N}. \nonumber
\end{align}
\end{proof}
\begin{lem}
\label{thm:positivity of shifted Jack}
For all $\mathbf{m},\mathbf{k} \in \mathscr{P}$, we have
\begin{equation}
\gamma_{\mathbf{k}}(\mathbf{m}-\rho)\geq 0.
\end{equation}
\end{lem}
\begin{proof}
Since $\gamma_{\mathbf{k}}(\mathbf{m}-\rho)=\frac{1}{d_{\mathbf{k}}}\left(\frac{n}{r}\right)_{\mathbf{k}}\binom{\mathbf{m}}{\mathbf{k}}$ 
and $d_{\mathbf{k}}, \left(\frac{n}{r}\right)_{\mathbf{k}}>0$, it suffices to show $\binom{\mathbf{m}}{\mathbf{k}}\geq 0$ for all $\mathbf{m},\mathbf{k} \in \mathscr{P}$. 
From \cite{OO}, generalized binomial coefficients are written as 
$$
\binom{\mathbf{m}}{\mathbf{k}}_{\frac{d}{2}}=\frac{P_{\mathbf{k}}^{\ast}\left(\mathbf{m};\frac{d}{2}\right)}{H_{\left(\frac{d}{2}\right)}(\mathbf{k})},
$$
where $P_{\mathbf{k}}^{\ast}\left(\mathbf{m};\frac{d}{2}\right)$ is the shifted Jack polynomial 
and $H_{(\frac{d}{2})}(\mathbf{k})>0$ is a deformation of the hook length. 
Moreover, by using (5.2) in \cite{OO}
$$
P_{\mathbf{k}}^{\ast}\left(\mathbf{m};\frac{d}{2}\right)
=\frac{\frac{d}{2}\text{-}\dim{\mathbf{m}/\mathbf{k}}}{\frac{d}{2}\text{-}\dim{\mathbf{m}}}|\mathbf{m}|(|\mathbf{m}|-1)\cdots(|\mathbf{m}|-|\mathbf{k}|+1).
$$
Further, the positivity of the generalized dimensions of the skew Young diagram, $\frac{d}{2}\text{-}\dim{\mathbf{m}/\mathbf{k}}$, follows from (5.1) of \cite{OO} and Chapter V\hspace{-.1em}I.\,6  of \cite{M}. 
Therefore, we obtain the positivity of the shifted Jack polynomial and the conclusion. 
\end{proof}
\begin{thm}
\label{thm:spherical Taylor expan lem}
{\rm{(1)}}\,
For $w \in \mathcal{D}, \mathbf{k} \in \mathscr{P}, \alpha \in \mathbb{C}$, we have
\begin{equation}
\label{eq:basic expansion 1}
(\alpha)_{\mathbf{k}}\Delta(e-w)^{-\alpha}\Phi_{\mathbf{k}}(w(e-w)^{-1})
=\sum_{\mathbf{x} \in \mathscr{P}}d_{\mathbf{x}}\frac{(\alpha)_{\mathbf{x}}}{\left(\frac{n}{r}\right)_{\mathbf{x}}}\gamma_{\mathbf{k}}(\mathbf{x}-\rho)\Phi_{\mathbf{x}}(w). 
\end{equation}
Here, we choose the branch of $\Delta (e-w)^{-\alpha}$ which takes the value $1$ at $w=0$.

\noindent
{\rm(2)}\,For $w \in V^{\mathbb{C}}, \mathbf{k} \in \mathscr{P}$, a $K$-invariant analytic function $e^{\tr{w}}\Phi_{\mathbf{k}}(w)$ has the following expansion. 
\begin{equation}
\label{eq:basic expansion 2}
e^{\tr{w}}\Phi_{\mathbf{k}}(w)=\sum_{\mathbf{x} \in \mathscr{P}}d_{\mathbf{x}}\frac{1}{\left(\frac{n}{r}\right)_{\mathbf{x}}}\gamma_{\mathbf{k}}(\mathbf{x}-\rho)\Phi_{\mathbf{x}}(w).
\end{equation}
\end{thm}
\begin{proof}
{\rm{(1)}}\,We take $w=u\sum_{j=1}^{r}\lambda_{j}c_{j} \in \mathcal{D}$ with $u \in U$ and $1>\lambda_{1}\geq \ldots\geq \lambda_{r}\geq 0$. 
By Lemmas\,\ref{thm:FK,Thm12.1.1} and \ref{thm:estimate of shifted Jack1}, there exist some $C>0$ and $N \in \mathbb{Z}_{\geq 0}$ such that
\begin{align}
\sum_{\mathbf{x} \in \mathscr{P}}\left|d_{\mathbf{x}}\frac{(\alpha)_{\mathbf{x}}}{\left(\frac{n}{r}\right)_{\mathbf{x}}}\gamma_{\mathbf{k}}(\mathbf{x}-\rho)\Phi_{\mathbf{x}}(w)\right|
&\leq \sum_{\mathbf{x} \in \mathscr{P}}d_{\mathbf{x}}\frac{|(\alpha)_{\mathbf{x}}|}{\left(\frac{n}{r}\right)_{\mathbf{x}}}|\gamma_{\mathbf{k}}(\mathbf{x}-\rho)||\Phi_{\mathbf{x}}(w)| \nonumber \\
&\leq C\prod_{l=1}^{r}\sum_{x_{l}\geq 0}\frac{(|\alpha|+d(r-1))_{x_{l}}}{x_{l}!}\left(x_{l}+\frac{d}{2}(r-1)\right)^{N}\lambda_{l}^{x_{l}}< \infty. \nonumber
\end{align}
Here, we remark 
$$
|\gamma_{\mathbf{k}}(\mathbf{x}-\rho)|\leq C\prod_{l=1}^{r}\left(\left|x_{l}-\frac{d}{4}(2l-r-1)\right|+\frac{d}{4}(r-1)\right)^{N}
\leq C\prod_{l=1}^{r}\left(x_{l}+\frac{d}{2}(r-1)\right)^{N}.
$$
Therefore, the right hand side of (\ref{eq:basic expansion 1}) converges absolutely. 
By analytic continuation, it is sufficient to show the assertion when $\Re{\alpha}>\frac{d}{2}(r-1)$ and $w \in \Omega \cap (e-\Omega) \subset \mathcal{D}$.
\begin{align}
\Phi_{\mathbf{k}}(\partial_{z})\sum_{\mathbf{x} \in \mathscr{P}}d_{\mathbf{x}}\frac{(\alpha)_{\mathbf{x}}}{\left(\frac{n}{r}\right)_{\mathbf{x}}}\Phi_{\mathbf{x}}(z)\Phi_{\mathbf{x}}(w)\bigg|_{z=e}
&=\sum_{\mathbf{x} \in \mathscr{P}}d_{\mathbf{x}}\frac{(\alpha)_{\mathbf{x}}}{\left(\frac{n}{r}\right)_{\mathbf{x}}}\Phi_{\mathbf{k}}(\partial_{z})\Phi_{\mathbf{x}}(z)|_{z=e}\Phi_{\mathbf{x}}(w) \nonumber \\
&=\sum_{\mathbf{x} \in \mathscr{P}}d_{\mathbf{x}}\frac{(\alpha)_{\mathbf{x}}}{\left(\frac{n}{r}\right)_{\mathbf{x}}}\gamma_{\mathbf{k}}(\mathbf{x}-\rho)\Phi_{\mathbf{x}}(w). \nonumber
\end{align}
On the other hand, 
\begin{align}
\Phi_{\mathbf{k}}(\partial_{z})\sum_{\mathbf{x} \in \mathscr{P}}d_{\mathbf{x}}\frac{(\alpha)_{\mathbf{x}}}{\left(\frac{n}{r}\right)_{\mathbf{x}}}\Phi_{\mathbf{x}}(z)\Phi_{\mathbf{x}}(w)\bigg|_{z=e}
&=\Phi_{\mathbf{k}}(\partial_{z})\Delta(w)^{-\alpha}\int_{K}\Delta(kw^{-1}-z)^{-\alpha}\,dk\bigg|_{z=e} \nonumber \\
&=\Delta(w)^{-\alpha}\int_{K}\Phi_{\mathbf{k}}(\partial_{z})\Delta(kw^{-1}-z)^{-\alpha}\big|_{z=e}\,dk. \nonumber 
\end{align}
Here, from $kw^{-1}-z \in T_{\Omega}$ for all $k \in K$ and Lemma\,\ref{thm:int formula}, 
\begin{align}
\Phi_{\mathbf{k}}(\partial_{z})\Delta(kw^{-1}-z)^{-\alpha}\big|_{z=e}
&=\Phi_{\mathbf{k}}(\partial_{z})\frac{1}{\Gamma_{\Omega}(\alpha)}\int_{\Omega}e^{-(x|kw^{-1}-z)}\Delta(x)^{\alpha}\Delta(x)^{-\frac{n}{r}}\,dx\bigg|_{z=e} \nonumber \\
&=\frac{1}{\Gamma_{\Omega}(\alpha)}\int_{\Omega}\Phi_{\mathbf{k}}(\partial_{z})e^{(x|z)}|_{z=e}e^{-(x|kw^{-1})}\Delta(x)^{\alpha}\Delta(x)^{-\frac{n}{r}}\,dx \nonumber \\
&=\frac{1}{\Gamma_{\Omega}(\alpha)}\int_{\Omega}\Phi_{\mathbf{k}}(x)e^{-(kx|(w^{-1}-e))}\Delta(x)^{\alpha}\Delta(x)^{-\frac{n}{r}}\,dx \nonumber \\
&=(\alpha)_{\mathbf{k}}\Delta(w^{-1}-e)^{-\alpha}\Phi_{\mathbf{k}}((w^{-1}-e)^{-1}). \nonumber
\end{align}
Therefore, 
\begin{align}
\Phi_{\mathbf{k}}(\partial_{z})\sum_{\mathbf{x} \in \mathscr{P}}d_{\mathbf{x}}\frac{(\alpha)_{\mathbf{x}}}{\left(\frac{n}{r}\right)_{\mathbf{x}}}\Phi_{\mathbf{x}}(z)\Phi_{\mathbf{x}}(w)\bigg|_{z=e}
&=\Delta(w)^{-\alpha}\int_{K}(\alpha)_{\mathbf{k}}\Delta(w^{-1}-e)^{-\alpha}\Phi_{\mathbf{k}}((w^{-1}-e)^{-1})\,dk \nonumber \\
&=(\alpha)_{\mathbf{k}}\Delta(e-w)^{-\alpha}\Phi_{\mathbf{k}}(w(e-w)^{-1}). \nonumber
\end{align}
{\rm{(2)}}\,Since 
the right hand side of (\ref{eq:basic expansion 2}) converges absolutely due to a similar argument of {\rm{(1)}}, we have
\begin{align}
e^{\tr{w}}\Phi_{\mathbf{k}}(w)
&=\lim_{\alpha \to \infty}(\alpha)_{\mathbf{k}}\Delta\left(e-\frac{w}{\alpha}\right)^{-\alpha}\Phi_{\mathbf{k}}\left(\frac{w}{\alpha}\left(e-\frac{w}{\alpha}\right)^{-1}\right) \nonumber \\
&=\sum_{\mathbf{x} \in \mathscr{P}}d_{\mathbf{x}}\frac{1}{\left(\frac{n}{r}\right)_{\mathbf{x}}}\gamma_{\mathbf{k}}(\mathbf{x}-\rho)\lim_{\alpha \to \infty}(\alpha)_{\mathbf{x}}\Phi_{\mathbf{x}}\left(\frac{w}{\alpha}\right) \nonumber \\
&=\sum_{\mathbf{x} \in \mathscr{P}}d_{\mathbf{x}}\frac{1}{\left(\frac{n}{r}\right)_{\mathbf{x}}}\gamma_{\mathbf{k}}(\mathbf{x}-\rho)\Phi_{\mathbf{x}}(w). \nonumber
\end{align}
\end{proof}
Next, we preview the gradient for a $\mathbb{C}$-valued and $V$-valued function $f$ on simple Euclidean Jordan algebra $V$. 
In this parts, we refer to \cite{Di}. 
For differentiable function $f:V\rightarrow \mathbb{R}$ and $x,u \in V$, we define the gradient, $\nabla f(x) \in V$, of $f$ by 
$$
(\nabla f(x)|u)=D_{u}f(x)=\frac{d}{dt}f(x+tu)\bigg|_{t=0}.
$$ 
For a $\mathbb{C}$-valued function $f=f_{1}+if_{2}$, we define $\nabla f=\nabla f_{1}+i\nabla f_{2}$. 
For $z=x+iy \in V^{\mathbb{C}}$, we define $D_{z}=D_{x}+iD_{y}$. 
Moreover, if $\{e_{1},\ldots,e_{n}\}$ is an orthonormal basis of $V$ and $x=\sum_{j=1}^{n}x_{j}e_{j} \in V^{\mathbb{C}}$, then 
$$
\nabla f(x)=\sum_{j=1}^{n}\frac{{\partial}f(x)}{{\partial}x_{j}}e_{j}.
$$
We remark that this expression is independent of the choice of an orthonormal basis of $V$. 

For a $V$-valued function $f:V\rightarrow V$ expressed by $f(x)=\sum_{j=1}^{r}f_{j}(x)e_{j}$, we define $\nabla f$ by
$$
\nabla f(x)=\sum_{j,l=1}^{n}\frac{{\partial}f_{j}(x)}{{\partial}x_{l}}e_{j}e_{l}.
$$
That is also well defined. 
Let us present some derivation formulas. 
\begin{lem}
\label{thm:derivation Lemma}
{\rm{(1)}}\,The product rule of differentiation: 
For $V$-valued function $f,h$, we have 
\begin{equation}
\tr{(\nabla(f(x)h(x)))}=\tr{(\nabla{f(x)})}h(x)+f(x)\tr{(\nabla{h(x)})}.
\end{equation}
For $\mathbb{C}$-valued functions $f,h$, 
\begin{equation}
\nabla(f(x)h(x))=(\nabla{f(x)})h(x)+f(x)(\nabla{h(x)}). 
\end{equation}
{\rm{(2)}}
\begin{equation}
\nabla{x}=\frac{n}{r}e.
\end{equation}
{\rm{(3)}}\,For any invertible element $x \in V^{\mathbb{C}}$, 
\begin{equation}
\tr{(x\nabla)}x^{-1}:=\tr{(x(\nabla{x^{-1}}))}=-\frac{n}{r}\tr{x^{-1}}.
\end{equation}
{\rm{(4)}}\,For $\beta \in \mathbb{C}$ and an invertible element $x \in V^{\mathbb{C}}$, 
\begin{equation}
\nabla (\Delta(x)^{\beta})=\beta \Delta(x)^{\beta}x^{-1}.
\end{equation}
\end{lem}
\noindent
{\rm{(1)}},\,{\rm{(2)}}, and {\rm{(4)}} are well known (see \cite{FK}, \cite{Di}, and \cite{FW1}). 
{\rm{(3)}} follows from {\rm{(1)}}, {\rm{(2)}}, and $\nabla(x x^{-1})=\nabla(e)=0$.

The following recurrence formulas for the spherical functions, some of which involve the gradient, are also well known (see \cite{Di} and \cite{FW1}).

Finally, we provide a Plancherel theorem, which is needed to investigate the MCJ polynomials. 
\begin{lem}
\label{thm:Plancherel,[FK] Thm 9.4.1}
Put
\begin{align}
L^{2}(\Omega)&:=\{\psi:\Omega  \longrightarrow \mathbb{C} \mid \|\psi\|_{\Omega}^{2}<\infty \}, \nonumber \\
  H^{2}(V)&:=\left\{\Psi:V \longrightarrow \mathbb{C} \mid \text{$\|\Psi\|_{V}^{2}<\infty$ and $\Psi$ is continued analytically to $H_{\Omega}$} \right. \nonumber \\
  {} & \quad \quad \quad \quad \quad \quad \quad \quad \quad \left. \text{as a holomorphic function which satisfies with} \right. \nonumber \\
  {} & \quad \quad \quad \quad \quad \quad \quad \quad \quad \left. \sup_{y \in \Omega}\frac{1}{(2\pi)^{n}}\int_{V}|\Psi(x+iy)|^{2}\,dx < \infty\right\}. \nonumber
\end{align}
Here, 
\begin{align}
\|\psi\|_{\Omega}^{2}:=\int_{\Omega}|\psi(u)|^{2}\,du, \,\,\,\,\,\,
\|\Psi\|_{V}^{2}:=\frac{1}{(2\pi)^{n}}\int_{V}|\Psi(t)|^{2}\,dt. \nonumber
\end{align}
The (inverse) Fourier transform of an integrable function, $\psi$, on $\Omega$ is defined as
\begin{equation}
(F^{-1}\psi)(t):=\int_{\Omega}e^{i(t|u)}\psi(u)\,du.
\end{equation}
We have
\begin{equation}
F^{-1}:L^{2}(\Omega) \,\,\,{\xrightarrow{\simeq}}\,\,\,H^{2}(V)\,\,\,\,({\rm{unitary}}).
\end{equation}
In particular, 
\begin{equation}
F^{-1}:L^{2}(\Omega)^{K} \,\,\,{\xrightarrow{\simeq}}\,\,\,H^{2}(V)^{K}\,\,\,\,({\rm{unitary}}).
\end{equation}
\end{lem}
\begin{proof}
From Theorem\,I\hspace{-.1em}X.4.1 in \cite{FK}, we have 
$$
\widetilde{F}^{-1}:L^{2}(\Omega)\,\,\,{\xrightarrow{\simeq}}\,\,\,H^{2}(H_{\Omega})\,\,\,\,({\rm{unitary}}),
$$
where 
\begin{align}
H^{2}(H_{\Omega})&:=\{\widetilde{\Psi}:H_{\Omega}:=V+i\Omega  \longrightarrow \mathbb{C} \mid \text{$\widetilde{\Psi}$ is analytic in $H_{\Omega}$ and $\|\widetilde{\Psi}\|_{H_{\Omega}}^{2}<\infty$}\}, \nonumber \\
\|\widetilde{\Psi}\|_{H_{\Omega}}^{2}&:=\sup_{y \in \Omega}\frac{1}{(2\pi)^{n}}\int_{V}|\widetilde{\Psi}(x+iy)|^{2}\,dx, \nonumber \\
\widetilde{F}^{-1}(\psi)(z)&:=\int_{\Omega}e^{i(z|u)}\psi(u)\,du. \nonumber 
\end{align}
Moreover, from Corollary\,I\hspace{-.1em}X.4.2 in \cite{FK}, for function $\widetilde{\Psi} \in H^{2}(H_{\Omega})$, $y \in \Omega$, we write $\widetilde{\Psi}_{y}(x):=\widetilde{\Psi}(x+iy)$; then,
$$
\lim_{y \to 0, y \in \Omega}\widetilde{\Psi}_{y}=\widetilde{\Psi}_{0},\,\,\,\,\,\widetilde{\Psi}_{0}(t):=\int_{\Omega}e^{i(t|u)}\psi(u)\,du=F^{-1}(\psi)(t),  
$$
exists in $L^{2}(V)$ and the map $\widetilde{\Psi} \mapsto \widetilde{\Psi}_{0}$ is an isometric embedding of $H^{2}(H_{\Omega})$ into $L^{2}(V)$.
Hence, the map $F^{-1}:\psi \mapsto \widetilde{\Psi}_{0}$ is unitary. 
The surjectivity of this map follows from the above facts and the definition of $H^{2}(V)$.


Furthermore, since the inverse Fourier transform $F^{-1}$ and the action of $K$ are commutative, the above unitary isomorphism also holds for the $K$-invariant spaces.  
\end{proof}

\subsection{Multivariate Laguerre polynomials and their unitary picture}
In this subsection, we promote the multivariate Laguerre polynomials 
and provide some fundamental lemmas based on \cite{FK}, \cite{FW1}.

First, we recall the multivariate Laguerre polynomials.  
Let $\alpha >\frac{n}{r}-1=\frac{d}{2}(r-1)$, $\mathbf{m} \in \mathscr{P}$. 
\\
{\bf{(1)}}\,\,$\psi_{\mathbf{m}}^{(\alpha)}$\,;\,Multivariate Laguerre polynomials (to multiply exponential)
\begin{align}
L^{2}_{\alpha}(\Omega)^{K}&:=\{\psi:\Omega \longrightarrow \mathbb{C} \mid \psi \text{ is $K$-invariant and } \|\psi\|_{\alpha,\Omega}^{2}<\infty\}, \nonumber \\
\|\psi\|_{\alpha,\Omega}^{2}&:=\frac{2^{r\alpha}}{\Gamma_{\Omega}(\alpha)}\int_{\Omega}|\psi(u)|^{2}\Delta(u)^{\alpha -\frac{n}{r}}\,du, \nonumber \\
\psi_{\mathbf{m}}^{(\alpha)}(u)&:=e^{-\tr{u}}L_{\mathbf{m}}^{(\alpha-\frac{n}{r})}(2u). \nonumber 
\end{align}
Here, $L_{\mathbf{m}}^{\left(\alpha -\frac{n}{r}\right)}(u)$ is the multivariate Laguerre polynomial defined by 
\begin{align}
L_{\mathbf{m}}^{\left(\alpha -\frac{n}{r}\right)}(u)
&:=d_{\mathbf{m}}\frac{(\alpha)_{\mathbf{m}}}{\left(\frac{n}{r}\right)_{\mathbf{m}}}
\sum_{\mathbf{k}\subset \mathbf{m}}(-1)^{|\mathbf{k}|}\binom{\mathbf{m}}{\mathbf{k}}\frac{1}{(\alpha)_{\mathbf{k}}}\Phi_{\mathbf{k}}(u) \nonumber \\
&=d_{\mathbf{m}}\frac{(\alpha)_{\mathbf{m}}}{\left(\frac{n}{r}\right)_{\mathbf{m}}}
\sum_{\mathbf{k}\subset \mathbf{m}}(-1)^{|\mathbf{k}|}d_{\mathbf{k}}\frac{\gamma_{\mathbf{k}}(\mathbf{m}-\rho)}{\left(\frac{n}{r}\right)_{\mathbf{k}}(\alpha)_{\mathbf{k}}}\Phi_{\mathbf{k}}(u). \nonumber
\end{align}
We remark that $\{\psi_{\mathbf{m}}^{(\alpha)}\}_{\mathbf{m} \in \mathscr{P}}$ form complete orthogonal basis of $L^{2}_{\alpha}(\Omega)^{K}$ and 
$$
\|\psi_{\mathbf{m}}^{(\alpha)}\|_{\alpha,\Omega}^{2}=d_{\mathbf{m}}\frac{(\alpha)_{\mathbf{m}}}{\left(\frac{n}{r}\right)_{\mathbf{m}}}.
$$
The multivariate Laguerre polynomials have the following generating function. 
\begin{lem}
\label{thm:generating fnc of Laguerre and MP}
\,For any $\alpha \in \mathbb{C},u \in \Omega$ and $z \in \mathcal{D}$, we have
\begin{equation}
\label{eq:generating fnc of Laguerre}
\sum_{\mathbf{m} \in \mathscr{P}}L_{\mathbf{m}}^{\left(\alpha -\frac{n}{r}\right)}(u)\Phi_{\mathbf{m}}(z)=\Delta(e-z)^{-\alpha}\int_{K}e^{-(ku|z(e-z)^{-1})}\,dk.
\end{equation}
\end{lem}
\begin{proof}
By referring to \cite{Dib} (see Proposition\,2.\,8), (\ref{eq:generating fnc of Laguerre}) holds for $\alpha >\frac{n}{r}-1=\frac{d}{2}(r-1)$. 
Moreover, the right hand side of (\ref{eq:generating fnc of Laguerre}) is well defined for any $\alpha \in \mathbb{C}$.  
Hence, by analytic continuation, it is sufficient to show the absolute convergence of the left hand side under the assumption. 
By Lemmas\,\ref{thm:ineq for generalized shifted factorial}, \ref{thm:FK,Thm12.1.1}, \ref{thm:positivity of shifted Jack} and \ref{thm:spherical Taylor expan lem}, 
\begin{align}
\sum_{\mathbf{m} \in \mathscr{P}}|L_{\mathbf{m}}^{\left(\alpha -\frac{n}{r}\right)}(u)\Phi_{\mathbf{m}}(z)|
&\leq \sum_{\mathbf{m} \in \mathscr{P}}\sum_{\mathbf{k}\subset \mathbf{m}}
\left|d_{\mathbf{m}}\frac{(\alpha)_{\mathbf{m}}}{\left(\frac{n}{r}\right)_{\mathbf{m}}}
\binom{\mathbf{m}}{\mathbf{k}}\frac{(-1)^{|\mathbf{k}|}}{(\alpha)_{\mathbf{k}}}\Phi_{\mathbf{k}}(u)\right|\Phi_{\mathbf{m}}(a_{1}) \nonumber \\
&\leq \sum_{\mathbf{k} \in \mathscr{P}}d_{\mathbf{k}}\frac{1}{\left(\frac{n}{r}\right)_{\mathbf{k}}}\frac{1}{(|\alpha|+d(r-1))_{\mathbf{k}}}\Phi_{\mathbf{k}}(u) \nonumber \\
{} & \quad \sum_{\mathbf{m} \in \mathscr{P}}d_{\mathbf{m}}\frac{(|\alpha|+d(r-1))_{\mathbf{m}}}{\left(\frac{n}{r}\right)_{\mathbf{m}}}\gamma_{\mathbf{k}}(\mathbf{m}-\rho)\Phi_{\mathbf{m}}(a_{1}) \nonumber \\
&=(1-a_{1})^{-r|\alpha|-dr(r-1)}\sum_{\mathbf{k} \in \mathscr{P}}d_{\mathbf{k}}\frac{1}{\left(\frac{n}{r}\right)_{\mathbf{k}}}\Phi_{\mathbf{k}}\left(\frac{a_{1}}{1-a_{1}}u\right) \nonumber \\
&=(1-a_{1})^{-r|\alpha|-dr(r-1)}e^{\frac{a_{1}}{1-a_{1}}\tr{u}}< \infty. \nonumber
\end{align}
\end{proof}
The multivariate Laguerre polynomials also satisfy with the following differential equation.  
\begin{lem}
\label{thm:differential Lemma for Laguerre}
Let us consider the operators $D_{\alpha}^{(1)}$. 
\begin{equation}
D_{\alpha}^{(1)}=\tr{(-u\nabla_{u}^{2}-\alpha\nabla_{u}+u-\alpha{e})}.
\end{equation}
We have
\begin{equation}
\label{eq:differential equation for Laguerre}
D_{\alpha}^{(1)}\psi_{\mathbf{m}}^{(\alpha)}(u)=2|\mathbf{m}|\psi_{\mathbf{m}}^{(\alpha)}(u).
\end{equation}
\end{lem}

\section{Multivariate circular Jacobi polynomials}
In the first subsection of this section, we introduce a new multivariate orthogonal polynomial, 
$\phi_{\mathbf{m}}^{(d)}(\sigma;\alpha,\nu)=\phi_{\mathbf{m}}^{(\alpha,\nu)}(\sigma)$, 
which is a 2-parameter deformation of the spherical polynomial. 
This is also regarded as a multivariate analogue of the circular Jacobi polynomial. 
Hence, we call this polynomial the {\it{multivariate circular Jacobi (MCJ) polynomial}} that degenerates to a 1-parameter deformation of the usual circular Jacobi polynomial, $\phi_{m}^{(\alpha)}(e^{i\theta})$, in the one variable case. 
Further, the weight function of its orthogonality relation coincides with the circular Jacobi ensemble defined by Bourgade et al.\,\cite{BNR}.

We derive a generating function of $\phi_{\mathbf{m}}^{(\alpha,\nu)}$ in subsection\,3.2 and a pseudo-differential equation for $\Psi_{\mathbf{m}}^{(\alpha,\nu)}$ in the subsection\,3.3. 
In case of the multiplicity $d=2$, we give a determinant formula for the MCJ polynomials in subsection\,3.4.  
Moreover, we study the one variable case in more detail in subsection\,3.5. 
Finally, we describe future work for $\phi_{\mathbf{m}}^{(\alpha,\nu)}$.

Unless otherwise specified, we have assumed $\alpha >\frac{n}{r}-1=\frac{d}{2}(r-1)$ and $\nu \in \mathbb{R}$ in this section.

\subsection{Definitions and orthogonality}
Following Section\,2, we introduce some function space, functions that become complete orthogonal bases and unitary transformations required to provide the MCJ polynomials.

\noindent
{\bf{(2)}}\,\,$\Psi_{\mathbf{m}}^{(\alpha,\nu)}$\,;\,Modified Fourier transform of $\psi_{\mathbf{m}}^{(\alpha)}$
\begin{align}
H^{2}_{\alpha,\nu}(V)^{K}&:=\{\Psi:V \longrightarrow \mathbb{C} \mid \Psi \in H^{2}(V) \text{ is $K$-invariant and } \|\Psi \|_{\alpha,\nu,V}^{2}<\infty\}, \nonumber \\
\|\Psi\|_{\alpha,\nu,V}^{2}&:=\frac{2^{r\alpha}}{\Gamma_{\Omega}(\alpha)}\left|\Gamma_{\Omega}\left(\frac{1}{2}\left(\alpha +\frac{n}{r}\right)+i\nu \right)\right|^{2}\|\Psi\|_{V}^{2} \nonumber \\
&=\frac{\widetilde{c_{0}}}{(2\pi)^{n}}\frac{2^{r\alpha}}{\Gamma_{\Omega}(\alpha)}\left|\Gamma_{\Omega}\left(\frac{1}{2}\left(\alpha +\frac{n}{r}\right)+i\nu \right)\right|^{2} 
\int_{\mathbb{R}^{r}}|\Psi(\lambda)|^{2}\prod_{1\leq p<q\leq r}|\lambda_{p}-\lambda_{q}|^{d}\,d\lambda_{1}\cdots{d\lambda_{r}}, \nonumber \\
\Psi_{\mathbf{m}}^{(\alpha,\nu)}(t)&:=\Delta (e-it)^{-\frac{1}{2}\left(\alpha +\frac{n}{r}\right)-i\nu}\widetilde{\Psi_{\mathbf{m}}^{(\alpha,\nu)}}(t), \nonumber \\
\widetilde{\Psi_{\mathbf{m}}^{(\alpha,\nu)}}(t)&:=d_{\mathbf{m}}\frac{(\alpha)_{\mathbf{m}}}{\left(\frac{n}{r}\right)_{\mathbf{m}}}\sum_{\mathbf{k}\subset \mathbf{m}}(-1)^{|\mathbf{k}|}\binom{\mathbf{m}}{\mathbf{k}}\frac{\left(\frac{1}{2}\left(\alpha +\frac{n}{r}\right)+i\nu\right)_{\mathbf{k}}}{(\alpha)_{\mathbf{k}}}\Phi_{\mathbf{k}}(2(e-it)^{-1}). \nonumber 
\end{align}
Here, $\lambda =\sum_{j=1}^{r}\lambda_{j}c_{j}$ and we choose the branch of $\Delta (e-it)^{-\frac{1}{2}\left(\alpha +\frac{n}{r}\right)-i\nu}$ which takes the value $1$ at $t=0$.

\noindent
{\bf{(3)}}\,\,$\phi_{\mathbf{m}}^{(\alpha,\nu)}$\,;\,MCJ polynomials
\begin{align}
H^{2}_{\alpha,\nu}(\Sigma)^{K}&:=\{\phi:\Sigma \longrightarrow \mathbb{C} \mid \phi \text{ is $K$-invariant and continued analytically to $\mathcal{D}$}  \nonumber \\
{} & \quad \quad \quad \quad \quad \quad \quad \quad \text{as a holomorphic function which satisfies with } \|\phi\|_{\alpha,\nu,\Sigma}^{2}<\infty \}, \nonumber \\
\|\phi\|_{\alpha,\nu,\Sigma}^{2}&:=\frac{1}{(2\pi)^{n}}\frac{1}{\Gamma_{\Omega}(\alpha)}\left|\Gamma_{\Omega}\left(\frac{1}{2}\left(\alpha +\frac{n}{r}\right)+i\nu \right)\right|^{2}
\int_{\Sigma}|\phi(\sigma)|^{2}|\Delta (e-\sigma)^{\frac{1}{2}(\alpha -\frac{n}{r})+i\nu}|^{2}\,d\mu(\sigma)  \nonumber \\
&=\frac{\widetilde{c_{0}}}{(2\pi)^{n}}\frac{1}{\Gamma_{\Omega}(\alpha)}\left|\Gamma_{\Omega}\left(\frac{1}{2}\left(\alpha +\frac{n}{r}\right)+i\nu \right)\right|^{2} \nonumber \\
{} & \quad \cdot \int_{\mathcal{S}^{r}}|\phi(e^{i\theta})|^{2}\prod_{j=1}^{r}|(1-e^{i\theta_{j}})^{\frac{1}{2}\left(\alpha -\frac{n}{r}\right)+i\nu}|^{2}
\prod_{1\leq p<q\leq r}|e^{i\theta_{p}}-e^{i\theta_{q}}|^{d}\,d\theta_{1}\cdots d\theta_{r}. \nonumber 
\end{align}
Here, we define the multivariate circular Jacobi polynomial by
\begin{equation}
\label{eq:def of MCJ}
\phi_{\mathbf{m}}^{(d)}(\sigma;\alpha,\nu)=\phi_{\mathbf{m}}^{(\alpha, \nu)}(\sigma):=d_{\mathbf{m}}\frac{(\alpha)_{\mathbf{m}}}{\left(\frac{n}{r}\right)_{\mathbf{m}}}
\sum_{\mathbf{k}\subset \mathbf{m}}(-1)^{|\mathbf{k}|}\binom{\mathbf{m}}{\mathbf{k}}\frac{\left(\frac{1}{2}\left(\alpha +\frac{n}{r}\right)+i\nu \right)_{\mathbf{k}}}{(\alpha)_{\mathbf{k}}}\Phi_{\mathbf{k}}(e-\sigma).
\end{equation}

The main purpose of this subsection is to show that these polynomials form the complete orthogonal basis of $H_{\alpha, \nu}^{2}(\Sigma)^{K}$ and to explicitly write their orthogonal relations. 
To achieve that purpose, we introduce a modified Fourier transform $\mathcal{F}_{\alpha}^{-1}$ for a function $\psi$ on $\Omega$ and the second inverse modified Cayley transform $\mathcal{C}_{\alpha,\nu}^{-1}$ 
as follows. 
\begin{align}
(\mathcal{F}_{\alpha}^{-1}\psi)(t)&:=\frac{1}{\Gamma_{\Omega}\left(\frac{1}{2}\left(\alpha +\frac{n}{r}\right)\right)}(F^{-1}(\Delta(u)^{\frac{1}{2}(\alpha -\frac{n}{r})}\psi))(t) \\
&=\frac{1}{\Gamma_{\Omega}\left(\frac{1}{2}\left(\alpha +\frac{n}{r}\right)\right)}\int_{\Omega}e^{i(t|u)}\psi(u)\Delta(u)^{\frac{1}{2}\left(\alpha -\frac{n}{r}\right)}\,du, \\
(\mathcal{C}_{\alpha,\nu}^{-1}\Psi)(\sigma)
&:=\Delta(e-ic(\sigma))^{\frac{1}{2}\left(\alpha +\frac{n}{r}\right)+i\nu}\Psi(c(\sigma))
=\Delta\left(\frac{e-\sigma}{2}\right)^{-\frac{1}{2}\left(\alpha +\frac{n}{r}\right)-i\nu}\!\!\!\!\Psi(c(\sigma)).
\end{align}
These give the following unitary isomorphisms. 
\begin{thm}
\label{thm:main theorem1}

\noindent
{\rm{(1)}}\,
\begin{align}
\begin{array}{ccccc}
\mathcal{F}_{\alpha,\nu}^{-1}:=\mathcal{F}_{\alpha+2i\nu}^{-1}:& L^{2}_{\alpha}(\Omega)^{K} & {\xrightarrow{\simeq}} & H_{\alpha, \nu}^{2}(V)^{K} & \text{{\rm{(unitary)}}.}  \\
 & \rotatebox{90}{$\in$} & & \rotatebox{90}{$\in$} & \nonumber \\
 & \psi_{\mathbf{m}}^{(\alpha)} & \longmapsto & \Psi_{\mathbf{m}}^{(\alpha,\nu)} & \nonumber
\end{array}
\end{align}
In particular, $\{\Psi_{\mathbf{m}}^{(\alpha,\nu)}\}_{\mathbf{m} \in \mathscr{P}}$ form the complete orthogonal basis of $H_{\alpha, \nu}^{2}(V)^{K}$ and for all $\mathbf{m},\mathbf{n} \in \mathscr{P}$, 
\begin{align}
\frac{1}{(2\pi)^{n}}\int_{V}\Psi_{\mathbf{m}}^{(\alpha,\nu)}(t)\overline{\Psi_{\mathbf{n}}^{(\alpha,\nu)}(t)}\,dt
=d_{\mathbf{m}}\frac{\Gamma_{\Omega}(\alpha+\mathbf{m})}{\left(\frac{n}{r}\right)_{\mathbf{m}}}\frac{1}{\left|\Gamma_{\Omega}\left(\frac{1}{2}\left(\alpha +\frac{n}{r}\right)+i\nu\right)\right|^{2}}\delta_{\mathbf{m}\mathbf{n}}.  
\end{align}
\noindent
{\rm{(2)}}\,
\begin{align}
\begin{array}{ccccc}
\mathcal{C}_{\alpha,\nu}^{-1}:& H_{\alpha, \nu}^{2}(V)^{K} & {\xrightarrow{\simeq}} & H_{\alpha, \nu}^{2}(\Sigma)^{K} & \text{{\rm{(unitary)}}.}  \\
 & \rotatebox{90}{$\in$} & & \rotatebox{90}{$\in$} & \nonumber \\
 & \Psi_{\mathbf{m}}^{(\alpha,\nu)} & \longmapsto & \phi_{\mathbf{m}}^{(\alpha,\nu)} & \nonumber
\end{array}
\end{align}
Furthermore, the MCJ polynomials form the complete orthogonal basis of $H_{\alpha, \nu}^{2}(\Sigma)^{K}$ and for all $\mathbf{m},\mathbf{n} \in \mathscr{P}$, 
\begin{align}
\label{eq:MCJ orthogonal}
&\frac{1}{(2\pi)^{n}}\int_{\Sigma}\phi_{\mathbf{m}}^{(\alpha,\nu)}(\sigma)\overline{\phi_{\mathbf{n}}^{(\alpha,\nu)}(\sigma)}|\Delta (e-\sigma)^{\frac{1}{2}(\alpha -\frac{n}{r})+i\nu}|^{2}\,d\mu(\sigma)  \nonumber \\
&=\frac{\widetilde{c_{0}}}{(2\pi)^{n}}\int_{\mathcal{S}^{r}}\phi_{\mathbf{m}}^{(\alpha,\nu)}(e^{i\theta})\overline{\phi_{\mathbf{n}}^{(\alpha,\nu)}(e^{i\theta})} 
\prod_{j=1}^{r}|(1-e^{i\theta_{j}})^{\frac{1}{2}\left(\alpha -\frac{n}{r}\right)+i\nu}|^{2}
\prod_{1\leq k<l\leq r}|e^{i\theta_{k}}-e^{i\theta_{l}}|^{d}\,d\theta_{1}\cdots d\theta_{r} \nonumber \\
&=d_{\mathbf{m}}\frac{\Gamma_{\Omega}(\alpha+\mathbf{m})}{\left(\frac{n}{r}\right)_{\mathbf{m}}}\frac{1}{\left|\Gamma_{\Omega}\left(\frac{1}{2}\left(\alpha +\frac{n}{r}\right)+i\nu\right)\right|^{2}}\delta_{\mathbf{m}\mathbf{n}}.
\end{align}
\end{thm}
\begin{proof}
{\rm{(1)}}\,Observing that 
\begin{align}
\begin{array}{cccc}
 L^{2}_{\alpha}(\Omega)^{K} & {\xrightarrow{\simeq}} & L^{2}(\Omega)^{K} & \text{{\rm{(unitary)}},} \nonumber \\
 \rotatebox{90}{$\in$} & & \rotatebox{90}{$\in$} & \nonumber \\
 \psi & \longmapsto & \left(\frac{2^{r\alpha}}{\Gamma_{\Omega}(\alpha)}\right)^\frac{1}{2}\Delta(u)^{\frac{1}{2}\left(\alpha -\frac{n}{r}\right)+i\nu}\psi & \nonumber
\end{array}
\end{align}
\begin{align}
\begin{array}{cccc}
 H_{\alpha, \nu}^{2}(V)^{K} & {\xrightarrow{\simeq}} & H^{2}(V)^{K} & \text{{\rm{(unitary)}}.} \nonumber \\
 \rotatebox{90}{$\in$} & & \rotatebox{90}{$\in$} & \nonumber \\
 \Psi & \longmapsto & \frac{2^{\frac{r\alpha}{2}}}{\Gamma_{\Omega}(\alpha)^\frac{1}{2}}\Gamma_{\Omega}\left(\frac{1}{2}\left(\alpha +\frac{n}{r}\right)+i\nu\right)\Psi  & \nonumber
\end{array}
\end{align}
In addition to using Lemma\,\ref{thm:Plancherel,[FK] Thm 9.4.1}, we immediately obtain the following unitary isomorphism $\mathcal{F}_{\alpha,\nu}^{-1}$. 
\begin{align}
\begin{array}{ccccccc}
L^{2}_{\alpha}(\Omega)^{K} \!\!\!\!& {\xrightarrow{\simeq}} &\!\!\! L^{2}(\Omega)^{K} \!\!\!& {\xrightarrow{\simeq}} &\!\!\! H^{2}(V)^{K} \!\!\!& {\xrightarrow{\simeq}} &\!\!\!\! H_{\alpha, \nu}^{2}(V)^{K}. \nonumber \\
\rotatebox{90}{$\in$} \!\!\!\!& &\!\!\! \rotatebox{90}{$\in$} \!\!\!& &\!\!\! \rotatebox{90}{$\in$} \!\!\!& &\!\!\!\!\!\!\!\! \rotatebox{90}{$\in$}  \nonumber \\
\psi \!\!\!\!& \longmapsto &\!\!\! \left(\frac{2^{r\alpha}}{\Gamma_{\Omega}(\alpha)}\right)^\frac{1}{2}\Delta(u)^{\frac{1}{2}\left(\alpha -\frac{n}{r}\right)+i\nu}\psi \!\!\!& \longmapsto &\!\!\! F^{-1}\left(\frac{2^{\frac{r\alpha}{2}}}{\Gamma_{\Omega}(\alpha)^\frac{1}{2}}\Delta(u)^{\frac{1}{2}\left(\alpha -\frac{n}{r}\right)+i\nu}\psi\right) \!\!\!& \longmapsto &\!\!\!\! \mathcal{F}_{\alpha,\nu}^{-1}(\psi)  \nonumber
\end{array}
\end{align}

Next, we evaluate the modified Fourier transform of $\psi_{\mathbf{m}}^{(\alpha)}$ which forms the complete orthogonal basis for $L^{2}_{\alpha}(\Omega)^{K}$. 
From Lemma\,\ref{thm:int formula}, we obtain 
\begin{equation}
\mathcal{F}_{\alpha,\nu}^{-1}(e^{-\tr{u}}\Phi_{\mathbf{k}})(t)
=\left(\frac{1}{2}\left(\alpha +\frac{n}{r}\right)+i\nu\right)_{\mathbf{k}}\Delta (e-it)^{-\frac{1}{2}\left(\alpha +\frac{n}{r}\right)-i\nu}\Phi_{\mathbf{k}}((e-it)^{-1}). \nonumber
\end{equation}
Hence,
\begin{align}
\mathcal{F}_{\alpha,\nu}^{-1}(\psi_{\mathbf{m}}^{(\alpha)})(t)
&=\frac{1}{\Gamma\left(\frac{1}{2}\left(\alpha +\frac{n}{r}\right)+i\nu\right)}F^{-1}(\Delta(u)^{\frac{1}{2}\left(\alpha -\frac{n}{r}\right)+i\nu}\psi_{\mathbf{m}}^{(\alpha)})(t) \nonumber \\
&=d_{\mathbf{m}}\frac{(\alpha)_{\mathbf{m}}}{\left(\frac{n}{r}\right)_{\mathbf{m}}}\sum_{\mathbf{k}\subset \mathbf{m}}(-2)^{|\mathbf{k}|}\binom{\mathbf{m}}{\mathbf{k}}\frac{1}{(\alpha)_{\mathbf{k}}} \nonumber \\
{} & \quad \cdot \frac{1}{\Gamma\left(\frac{1}{2}\left(\alpha +\frac{n}{r}\right)+i\nu\right)}\int_{\Omega}e^{i(t|u)}\Delta(u)^{\frac{1}{2}\left(\alpha -\frac{n}{r}\right)+i\nu}e^{-\tr{u}}\Phi_{\mathbf{k}}(u)\,du \nonumber \\
&=\Psi_{\mathbf{m}}^{(\alpha, \nu)}(t).  \nonumber 
\end{align}
%
Finally, we have
\begin{align}
d_{\mathbf{m}}\frac{(\alpha)_{\mathbf{m}}}{\left(\frac{n}{r}\right)_{\mathbf{m}}}\delta_{\mathbf{m}\mathbf{n}}
&=(\psi_{\mathbf{m}}^{(\alpha)},\psi_{\mathbf{n}}^{(\alpha)})_{\alpha, \Omega} \nonumber \\
&=(\mathcal{F}_{\alpha,\nu}^{-1}(\psi_{\mathbf{m}}^{(\alpha)}),\mathcal{F}_{\alpha,\nu}^{-1}(\psi_{\mathbf{n}}^{(\alpha)}))_{\alpha, \nu, V} \nonumber \\
&=\frac{2^{r\alpha}}{\Gamma_{\Omega}(\alpha)}\frac{\left|\Gamma_{\Omega}\left(\frac{1}{2}\left(\alpha +\frac{n}{r}\right)\right)+i\nu \right|^{2}}{(2\pi)^{n}}
\int_{V}\Psi_{\mathbf{m}}^{(\alpha,\nu)}(t)\overline{\Psi_{\mathbf{n}}^{(\alpha,\nu)}(t)}\,dt. \nonumber 
\end{align}
{\rm{(2)}}\,We remark that the inverse Cayley transform $c^{-1}$ is a holomorphic bijection of $H_{\Omega}$ onto $\mathcal{D}$ and the inverse map of $\mathcal{C}_{\alpha,\nu}^{-1}$ is given by   
\begin{equation}
(\mathcal{C}_{\alpha,\nu}\phi)(t)
:=\Delta\left(\frac{e-c^{-1}(t)}{2}\right)^{\frac{1}{2}\left(\alpha +\frac{n}{r}\right)+i\nu}\phi(c^{-1}(t))
=\Delta(e-it)^{-\frac{1}{2}\left(\alpha +\frac{n}{r}\right)-i\nu}\phi(c^{-1}(t)). 
\end{equation}
Hence, 
since $\{\Psi_{\mathbf{m}}^{(\alpha,\nu)}\}_{\mathbf{m} \in \mathscr{P}}$ form the complete orthogonal basis of $H_{\alpha, \nu}^{2}(V)^{K}$, 
it is sufficient to show the statement for $\{\Psi_{\mathbf{m}}^{(\alpha,\nu)}\}_{\mathbf{m} \in \mathscr{P}}$ and $\{\phi_{\mathbf{m}}^{(\alpha,\nu)}\}_{\mathbf{m} \in \mathscr{P}}$. 

First, by Theorem\,I\hspace{-.1em}X.4.1 in \cite{FK}, 
$$
\|\Psi_{\mathbf{m}}^{(\alpha,\nu)}\|_{H_{\Omega}}^{2}
=\sup_{y \in \Omega}\frac{1}{(2\pi)^{n}}\int_{V}|\Psi_{\mathbf{m}}^{(\alpha,\nu)}(x+iy)|^{2}\,dx
=\frac{\Gamma_{\Omega}(\alpha)}{2^{r\alpha}\left|\Gamma_{\Omega}\left(\frac{1}{2}\left(\alpha +\frac{n}{r}\right)\right)+i\nu \right|^{2}}\|\psi_{\mathbf{m}}^{(\alpha)}\|_{\Omega}^{2}<\infty.
$$
Thus, we need not to consider the condition for $\|\Psi_{\mathbf{m}}^{(\alpha,\nu)}\|_{H_{\Omega}}^{2}$.  
Furthermore, by the definition, we have
$$
(\mathcal{C}_{\alpha,\nu}^{-1}\Psi_{\mathbf{m}}^{(\alpha,\nu)})(\sigma)
=\Delta\left(\frac{e-\sigma}{2}\right)^{-\frac{1}{2}\left(\alpha +\frac{n}{r}\right)-i\nu}\!\!\!\!
\Delta(e-i{c(\sigma)})^{-\frac{1}{2}\left(\alpha +\frac{n}{r}\right)-i\nu}\widetilde{\Psi_{\mathbf{m}}^{(\alpha,\nu)}}(c(\sigma))
=\phi_{\mathbf{m}}^{(\alpha,\nu)}(\sigma),   
$$
and from (\ref{eq:int for Shilov boundary2}) of Lemma\,\ref{thm:int for Shilov boundary}, we have
\begin{align}
\int_{V}|\Psi_{\mathbf{m}}^{(\alpha,\nu)}(t)|^{2}\,dt
&=2^{n}\int_{\Sigma}|\Delta(e-\sigma)^{-\frac{n}{r}}|^{2}|\Psi_{\mathbf{m}}^{(\alpha,\nu)}(c(\sigma))|^{2}\,d\mu(\sigma) \nonumber \\
&=2^{n}\int_{\Sigma}|\Delta(e-\sigma)^{-\frac{n}{r}}|^{2}\left|\Delta\left(\frac{e-\sigma}{2}\right)^{\frac{1}{2}\left(\alpha+\frac{n}{r}\right)+i\nu}\right|^{2}
|(\mathcal{C}_{\alpha,\nu}^{-1}\Psi_{\mathbf{m}}^{(\alpha,\nu)})(\sigma)|^{2}\,d\mu(\sigma) \nonumber \\
&=2^{-r\alpha}\int_{\Sigma}|\phi_{\mathbf{m}}^{(\alpha,\nu)}(\sigma)|^{2}|\Delta(e-\sigma)^{\frac{1}{2}\left(\alpha-\frac{n}{r}\right)+i\nu}|^{2}\,d\mu(\sigma). \nonumber
\end{align}
Therefore, we obtain 
$$
(\phi_{\mathbf{m}}^{(\alpha,\nu)},\phi_{\mathbf{n}}^{(\alpha,\nu)})_{\alpha, \nu, \Sigma}
=(\mathcal{C}_{\alpha,\nu}^{-1}(\Psi_{\mathbf{m}}^{(\alpha,\nu)}),\mathcal{C}_{\alpha,\nu}^{-1}(\Psi_{\mathbf{n}}^{(\alpha,\nu)}))_{\alpha, \nu, \Sigma}
=(\Psi_{\mathbf{m}}^{(\alpha,\nu)},\Psi_{\mathbf{n}}^{(\alpha,\nu)})_{\alpha, \nu, V}
=d_{\mathbf{m}}\frac{(\alpha)_{\mathbf{m}}}{\left(\frac{n}{r}\right)_{\mathbf{m}}}\delta_{\mathbf{m}\mathbf{n}}.
$$
(\ref{eq:MCJ orthogonal}) follows from (\ref{eq:int for Shilov boundary3}) of Lemma\,\ref{thm:int for Shilov boundary} immediately. 
\end{proof}

\begin{rmk}
\noindent
{\rm{(1)}}\,Our multivariate orthogonal polynomials are not equal BC-type multivariate Jacobi polynomials. 
Actually, the weight function of the left hand side for (\ref{eq:MCJ orthogonal})
$$
\prod_{j=1}^{r}(1-e^{i\theta_{j}})^{\frac{1}{2}\left(\alpha -\frac{n}{r}\right)+i\nu}(1-e^{-i\theta_{j}})^{\frac{1}{2}\left(\alpha -\frac{n}{r}\right)-i\nu}
\prod_{1\leq p<q\leq r}|e^{i\theta_{p}}-e^{i\theta_{q}}|^{d}
$$
coincides with the circular Jacobi ensemble defined by \cite{BNR}, which is not the BC-type (Weyl group invariant) weight function.

\noindent
{\rm{(2)}}\,When $\alpha =\frac{n}{r}, \nu =0$, (\ref{eq:def of MCJ}) and (\ref{eq:MCJ orthogonal}) degenerate to 
\begin{align}
&\phi_{\mathbf{m}}^{(\frac{n}{r},0)}(e^{i\theta})=
d_{\mathbf{m}}\sum_{\mathbf{k}\subset \mathbf{m}}(-1)^{|\mathbf{k}|}\binom{\mathbf{m}}{\mathbf{k}}\Phi_{\mathbf{k}}(e-e^{i\theta })
=d_{\mathbf{m}}\Phi_{\mathbf{m}}(e^{i\theta}),  \\
&\frac{1}{(2\pi)^{n}}\int_{\Sigma}\phi_{\mathbf{m}}^{(\frac{n}{r},0)}(\sigma)\overline{\phi_{\mathbf{n}}^{(\frac{n}{r},0)}(\sigma)}\,d\mu(\sigma)  \nonumber \\
&=\frac{\widetilde{c_{0}}}{(2\pi)^{n}}\int_{\mathcal{S}^{r}}\phi_{\mathbf{m}}^{(\frac{n}{r},0)}(e^{i\theta})\overline{\phi_{\mathbf{n}}^{(\frac{n}{r},0)}(e^{i\theta})} 
\prod_{1\leq p<q\leq r}|e^{i\theta_{p}}-e^{i\theta_{q}}|^{d}\,d\theta_{1}\cdots d\theta_{r} 
=d_{\mathbf{m}}\frac{1}{\Gamma_{\Omega}\left(\frac{n}{r}\right)}\delta_{\mathbf{m}\mathbf{n}}. 
\end{align}
Therefore, $\phi_{\mathbf{m}}^{(\alpha,\nu)}(e^{i\theta})$ is regarded as a 2-parameter deformation of the spherical polynomial. 

As a generalization of spherical polynomial $\Phi_{\mathbf{m}}^{(d)}$, the Jack polynomial $P_{\mathbf{m}}^{(\frac{2}{d})}$ which is a generalization for multiplicity $d$ is well known (see \cite{M}, Chapter\,V\hspace{-.1em}I). 
This multivariate special orthogonal polynomial system is derived as the simultaneous eigenfunctions of some commuting differential operators.
On the other hand, using the unitary picture, 
we obtain another extension $\phi_{m}^{(\alpha,\nu)}$, which is different from the Jack polynomial, that is, instead of the multiplicity $d$, we consider  deformations for real 2-parameters $\alpha$ and $\nu$. 

\noindent
{\rm{(3)}}\,We remark that the relations between the MCJ polynomials and the Multivariate Meixner-Pollaczek polynomials. 
The Multivariate Meixner-Pollaczek polynomials $P_{\mathbf{m}}^{(\alpha)}(\mathbf{s};\theta)$ are introduced by Faraut-Wakayama \cite{FW1}. 
\begin{align}
P_{\mathbf{m}}^{(\alpha)}(\mathbf{s};\theta )
&:=e^{i|\mathbf{m}|\theta}d_{\mathbf{m}}\frac{(2\alpha)_{\mathbf{m}}}{\left(\frac{n}{r}\right)_{\mathbf{m}}}
\sum_{\mathbf{k}\subset \mathbf{m}}\binom{\mathbf{m}}{\mathbf{k}}\frac{\gamma_{\mathbf{k}}(-i\mathbf{s}-\alpha)}{(2\alpha)_{\mathbf{k}}}(1-e^{-2i\theta})^{|\mathbf{k}|} \nonumber \\
&=e^{i|\mathbf{m}|\theta}d_{\mathbf{m}}\frac{(2\alpha)_{\mathbf{m}}}{\left(\frac{n}{r}\right)_{\mathbf{m}}}
\sum_{\mathbf{k}\subset \mathbf{m}}d_{\mathbf{k}}\frac{\gamma_{\mathbf{k}}(\mathbf{m}-\rho)\gamma_{\mathbf{k}}(-i\mathbf{s}-\alpha)}{\left(\frac{n}{r}\right)_{\mathbf{k}}(2\alpha)_{\mathbf{k}}}(1-e^{-2i\theta})^{|\mathbf{k}|}. \nonumber
\end{align}
For $\alpha >\frac{n}{r}-1, 0<\theta <2\pi$, this system has the following orthogonality relations. 
\begin{align}
\frac{1}{(2\pi)^{r}}\int_{\mathbb{R}^{r}}q_{\mathbf{m}}^{(\alpha,\theta)}(\mathbf{s})\overline{q_{\mathbf{n}}^{(\alpha,\theta)}(\mathbf{s})}
e^{(2\theta -\pi)|\mathbf{s}|}\left|\Gamma_{\Omega}\left(i\mathbf{s}+\frac{\alpha}{2}+\rho \right)\right|^{2}\,\frac{m(d\mathbf{s})}{|c(i\mathbf{s})|^{2}}
=d_{\mathbf{m}}\frac{\Gamma_{\Omega}(\alpha+\mathbf{m})}{\left(\frac{n}{r}\right)_{\mathbf{m}}(2\sin{\theta})^{r\alpha}}\delta_{\mathbf{m}\mathbf{n}}. \nonumber
\end{align}
Here, $q_{\mathbf{m}}^{(\alpha,\theta)}(\mathbf{s})=e^{-i|\mathbf{m}|\theta}P_{\mathbf{m}}^{(\frac{\alpha}{2})}(\mathbf{s};\theta)$ and $m$ is the Lebesgue measure on $\mathbb{R}^{r}$. 
From (\ref{eq:gamma special case}), for any $\theta \in \mathbb{R}$ and $\mathbf{m} \in \mathscr{P}$, we have 
\begin{equation}
\label{eq:CJacobi and MP}
\phi_{\mathbf{m}}^{(\alpha,\nu)}(e^{i\theta}e)=q_{\mathbf{m}}^{\left(\alpha, -\frac{\theta}{2}\right)}\left(\nu -i\left(\frac{n}{2r}+\rho \right)\right).
\end{equation}
\end{rmk}
\subsection{Generating function}

By using these unitary isomorphisms, we present the generating functions of MCJ polynomials from (\ref{eq:generating fnc of Laguerre}).
\begin{thm}
We assume $z=u\sum_{j=1}^{r}a_{j}c_{j} \in \mathcal{D}$ with $u \in U$, $1>a_{1}\geq \ldots\geq a_{r}\geq 0$ and 
$a_{1} < \frac{1}{3}$. 

\noindent
{\rm{(1)}}\,For all $t \in V$, 
\begin{equation}
\label{eq:generating fnc of Psi}
\sum_{\mathbf{m} \in \mathscr{P}}\Psi_{\mathbf{m}}^{(\alpha,\nu)}(t)\Phi_{\mathbf{m}}(z)=\Delta(e-z)^{-\alpha}\int_{K}\Delta((e+z)(e-z)^{-1}-ikt)^{-\frac{1}{2}\left(\alpha +\frac{n}{r}\right)-i\nu}\,dk.
\end{equation}

\noindent
{\rm{(2)}}\,For any $\sigma \in \Sigma$,
\begin{align}
\label{eq:generating fnc of MCJ}
\sum_{\mathbf{m} \in \mathscr{P}}\phi_{\mathbf{m}}^{(\alpha,\nu)}(\sigma)\Phi_{\mathbf{m}}(z)&=\Delta(e-z)^{-\alpha} \nonumber \\
& \quad \cdot \Delta(e-\sigma)^{-\frac{1}{2}\left(\alpha +\frac{n}{r}\right)-i\nu}\int_{K}\Delta(z(e-z)^{-1}+k(e-\sigma)^{-1})^{-\frac{1}{2}\left(\alpha +\frac{n}{r}\right)-i\nu}\,dk.
\end{align}
\end{thm}

\begin{proof}
{\rm{(1)}}\,From a similar argument in the proof of Lemma\,\ref{thm:generating fnc of Laguerre and MP}, 
\begin{align}
\mathcal{F}_{\alpha,\nu}^{-1}\left(\sum_{\mathbf{m} \in \mathscr{P}}|\psi_{\mathbf{m}}^{(\alpha)}(u)\Phi_{\mathbf{m}}(z)|\right)(t) 
&\leq (1-a_{1})^{-r|\alpha|-dr(r-1)}\mathcal{F}_{\alpha,\nu}^{-1}(e^{-\frac{1-3a_{1}}{1-a_{1}}\tr{u}})(t) \nonumber \\
&=(1-a_{1})^{-r|\alpha|-dr(r-1)}\Delta\left(\frac{1-3a_{1}}{1-a_{1}}e-it\right)^{-\frac{1}{2}\left(\alpha+\frac{n}{r}\right)-i\nu}<\infty. \nonumber 
\end{align}
Hence, the exchange of integration and summation is justified and we obtain
\begin{align}
\sum_{\mathbf{m} \in \mathscr{P}}\Psi_{\mathbf{m}}^{(\alpha,\nu)}(t)\Phi_{\mathbf{m}}(z)
&=\sum_{\mathbf{m} \in \mathscr{P}}\mathcal{F}_{\alpha,\nu}^{-1}(\psi_{\mathbf{m}}^{(\alpha)})(t)\Phi_{\mathbf{m}}(z) \nonumber \\
&=\mathcal{F}_{\alpha,\nu}^{-1}\left(\sum_{\mathbf{m} \in \mathscr{P}}\psi_{\mathbf{m}}^{(\alpha)}(u)\Phi_{\mathbf{m}}(z)\right)(t) \nonumber \\
&=\Delta(e-z)^{-\alpha}\mathcal{F}_{\alpha,\nu}^{-1}\left(\int_{K}e^{-(ku|(e+z)(e-z)^{-1})}\,dk\right)(t). \nonumber
\end{align}
Moreover, by Lemma\,\ref{thm:int formula}, 
\begin{align}
\mathcal{F}_{\alpha,\nu}^{-1}\left(\int_{K}e^{-(ku|(e+z)(e-z)^{-1})}\,dk\right)(t)
&=\frac{1}{\Gamma_{\Omega}\left(\frac{1}{2}\left(\alpha+\frac{n}{r}\right)+i\nu\right)} \nonumber \\
{} & \quad \cdot \int_{\Omega}e^{i(t|u)}\Delta(u)^{\frac{1}{2}\left(\alpha -\frac{n}{r}\right)+i\nu}\int_{K}e^{-(ku|(e+z)(e-z)^{-1})}\,dkdu \nonumber \\
&=\frac{1}{\Gamma_{\Omega}\left(\frac{1}{2}\left(\alpha+\frac{n}{r}\right)+i\nu\right)} \nonumber \\
{} & \quad \cdot \int_{K}\int_{\Omega}e^{-(u|k(e+z)(e-z)^{-1}-it)}\Delta(u)^{\frac{1}{2}\left(\alpha -\frac{n}{r}\right)+i\nu}\,dudk \nonumber \\
&=\int_{K}\Delta(k(e+z)(e-z)^{-1}-it)^{-\frac{1}{2}\left(\alpha +\frac{n}{r}\right)-i\nu}\,dk \nonumber \\
&=\int_{K}\Delta((e+z)(e-z)^{-1}-ikt)^{-\frac{1}{2}\left(\alpha +\frac{n}{r}\right)-i\nu}\,dk. \nonumber 
\end{align}

\noindent
{\rm{(2)}}\,Applying the modified Cayley transform $\mathcal{C}_{\alpha,\nu}^{-1}$ to (\ref{eq:generating fnc of Psi}), we obtain 
\begin{align}
\sum_{\mathbf{m} \in \mathscr{P}}\phi_{\mathbf{m}}^{(\alpha,\nu)}(\sigma)\Phi_{\mathbf{m}}(z)
&=\Delta(e-z)^{-\alpha}\mathcal{C}_{\alpha,\nu}^{-1}\left(\int_{K}\Delta((e+z)(e-z)^{-1}-ikt)^{-\frac{1}{2}\left(\alpha +\frac{n}{r}\right)-i\nu}\,dk\right)(\sigma) \nonumber \\
&=\Delta(e-z)^{-\alpha}\Delta\left(\frac{e-\sigma}{2}\right)^{-\frac{1}{2}\left(\alpha +\frac{n}{r}\right)-i\nu} \nonumber \\
{} & \quad \cdot \int_{K}\Delta((e+z)(e-z)^{-1}-ikc(\sigma))^{-\frac{1}{2}\left(\alpha +\frac{n}{r}\right)-i\nu}\,dk. \nonumber
\end{align}
Since 
$$
(e+z)(e-z)^{-1}-ikc(\sigma)=2z(e-z)^{-1}+2k(e-\sigma)^{-1}, 
$$
we have the conclusion. 

\end{proof}
\subsection{Differential equation for $\Psi_{\mathbf{m}}^{(\alpha, \nu)}$}
Considering some pseudo-differential operator that is defined by 
\begin{equation}
\tr{(\nabla_{t}^{-1})}e^{(t|u)}=\tr{(u^{-1})}e^{(t|u)}\,\,\,\,\,\,(\text{any $t \in V, u \in \Omega$}), 
\end{equation}
we obtain the explicit (pseudo-) differential equations for $\Psi_{\mathbf{m}}^{(\alpha, \nu)}$ as follows. 
\begin{thm}
\label{thm:Differential equation for Psi}
The operator $D_{\alpha,\nu}^{(2)}$ on $V$ is defined by the relation $D_{\alpha,\nu}^{(2)}\mathcal{F}_{\alpha,\nu}^{-1}=\mathcal{F}_{\alpha,\nu}^{-1}D_{\alpha}^{(1)}$. 
Then, we obtain 
\begin{align}
\label{eq:Differential equation for Psi}
D_{\alpha,\nu}^{(2)}=\tr\left(-i(e+t^{2})\nabla_{t}+\left(2\nu-i\frac{n}{r}\right)t-\alpha{e}+i\left(\frac{1}{4}\left(\alpha -\frac{n}{r}\right)^{2}+\nu^{2} \right)\nabla_{t}^{-1}\right), 
\end{align}
and  
\begin{equation}
\label{eq:Differential equation for Psi 2}
D_{\alpha,\nu}^{(2)}\Psi_{\mathbf{m}}^{(\alpha, \nu)}(t)=2|\mathbf{m}|\Psi_{\mathbf{m}}^{(\alpha, \nu)}(t).
\end{equation}
\end{thm}
\begin{proof}
From Theorem\,\ref{thm:main theorem1}, to prove (\ref{eq:Differential equation for Psi}), it suffices to show the relation for complete orthogonal basis $\psi_{\mathbf{m}}^{(\alpha)}$ 
for $L^{2}_{\alpha}(\Omega)^{K}$
$$
(\mathcal{F}_{\alpha,\nu}^{-1}D_{\alpha}^{(1)}\psi_{\mathbf{m}}^{(\alpha)})(t)=\widetilde{D_{\alpha,\nu}^{(2)}}((\mathcal{F}_{\alpha,\nu}^{-1}\psi_{\mathbf{m}}^{(\alpha)})(t)),
$$
where $D_{\alpha,\nu}^{(2)}$ is the operator on the right hand side of (\ref{eq:Differential equation for Psi}).

By the very definition of the modified Fourier transform $\mathcal{F}_{\alpha,\nu}^{-1}$ and the inner product of $L^{2}_{\alpha}(\Omega)$, we can write 
\begin{equation}
(\mathcal{F}_{\alpha,\nu}^{-1}\psi)(t)=(e^{i(t|u)}\Delta(u)^{-\frac{1}{2}\left(\alpha -\frac{n}{r}\right)+i\nu}|\overline{\psi})_{L^{2}_{\alpha}(\Omega)},
\end{equation}
and from Lemma\,3.13 in \cite{FW1}, $D_{\alpha}^{(1)}=\overline{D_{\alpha}^{(1)}}$ is a self-adjoint operator with respect to the measure $\Delta(u)^{\alpha -\frac{n}{r}}\,du$. 
Hence, 
\begin{align}
(\mathcal{F}_{\alpha,\nu}^{-1}D_{\alpha}^{(1)}\psi_{\mathbf{m}}^{(\alpha)})(t)
&=(e^{i(t|u)}\Delta(u)^{-\frac{1}{2}\left(\alpha -\frac{n}{r}\right)+i\nu}|\overline{D_{\alpha}^{(1)}\psi_{\mathbf{m}}^{(\alpha)}})_{L^{2}_{\alpha}(\Omega)} \nonumber \\
&=(D_{\alpha}^{(1)}(e^{i(t|u)}\Delta(u)^{-\frac{1}{2}\left(\alpha -\frac{n}{r}\right)+i\nu})|\overline{\psi_{\mathbf{m}}^{(\alpha)}})_{L^{2}_{\alpha}(\Omega)}. \nonumber
\end{align}
Furthermore, based on Lemma\,\ref{thm:derivation Lemma}, let us perform  
\begin{align}
\tr(u\nabla_{u}^{2})(e^{i(t|u)}\Delta(u)^{-\frac{1}{2}\left(\alpha -\frac{n}{r}\right)+i\nu})
&=\tr(u(\nabla_{u}^{2}e^{i(t|u)}))\Delta(u)^{-\frac{1}{2}\left(\alpha -\frac{n}{r}\right)+i\nu} \nonumber \\
{} & \quad +2\tr(u(\nabla_{u}e^{i(t|u)})(\nabla_{u}\Delta(u)^{-\frac{1}{2}\left(\alpha -\frac{n}{r}\right)+i\nu})) \nonumber \\
{} & \quad +e^{i(t|u)}\tr(u\nabla_{u}^{2}\Delta(u)^{-\frac{1}{2}\left(\alpha -\frac{n}{r}\right)+i\nu}) \nonumber \\
&=e^{i(t|u)}\Delta(u)^{-\frac{1}{2}\left(\alpha -\frac{n}{r}\right)+i\nu}\left\{\tr(-ut^{2})-i\left(\!\left(\alpha -\frac{n}{r}\right)\!-2i\nu\!\right)\!\tr(t) \right. \nonumber \\
{} & \quad \left.+ \left(\frac{1}{2}\left(\alpha -\frac{n}{r}\right)-i\nu \right)\left(\frac{1}{2}\left(\alpha +\frac{n}{r}\right)-i\nu \right)\tr(u^{-1})\right\}, \nonumber 
\end{align}
and 
\begin{align}
\tr(\alpha \nabla_{u})(e^{i(t|u)}\Delta(u)^{-\frac{1}{2}\left(\alpha -\frac{n}{r}\right)+i\nu})
&=\alpha\tr(\nabla_{u}e^{i(t|u)})\Delta(u)^{-\frac{1}{2}\left(\alpha -\frac{n}{r}\right)+i\nu} \nonumber \\
{} & \quad +\alpha e^{i(t|u)}\tr(\nabla_{u}\Delta(u)^{-\frac{1}{2}\left(\alpha -\frac{n}{r}\right)+i\nu}) \nonumber \\
&=\tr\!\left(\!i\alpha{t}-\alpha\!\left(\frac{1}{2}\left(\alpha -\frac{n}{r}\right)-i\nu \right)u^{-1}\!\right)\!e^{i(t|u)}\Delta(u)^{-\frac{1}{2}\left(\alpha -\frac{n}{r}\right)+i\nu}. \nonumber
\end{align}
Here, we remark that for any $p \in \mathbb{Z}_{\geq 0}$ $\tr((\nabla_{u}^{p}e^{i(t|u)}))=\tr((it)^{p})e^{i(t|u)}$ and 
\begin{align}
\tr(u\nabla_{u}^{2}\Delta(u)^{-\frac{1}{2}\left(\alpha -\frac{n}{r}\right)+i\nu})
&=\left(-\frac{1}{2}\left(\alpha -\frac{n}{r}\right)+i\nu \right) \nonumber \\
{} & \quad \left\{\tr(u(\nabla_{u}u^{-1}))\Delta(u)^{-\frac{1}{2}\left(\alpha -\frac{n}{r}\right)+i\nu}
+\tr(\nabla_{u}\Delta(u)^{-\frac{1}{2}\left(\alpha -\frac{n}{r}\right)+i\nu})  \right\} \nonumber \\
&=\left(\frac{1}{2}\left(\alpha -\frac{n}{r}\right)-i\nu \right)\left(\frac{1}{2}\left(\alpha +\frac{n}{r}\right)-i\nu \right)
\tr(u^{-1})\Delta(u)^{-\frac{1}{2}\left(\alpha -\frac{n}{r}\right)+i\nu} \nonumber 
\end{align}
and 
\begin{align}
\tr((\nabla_{u}e^{i(t|u)})(u\nabla_{u}\Delta(u)^{-\frac{1}{2}\left(\alpha -\frac{n}{r}\right)+i\nu}))
&=\left(-\frac{1}{2}\left(\alpha -\frac{n}{r}\right)+i\nu\right)\tr(it)e^{i(t|u)}\Delta(u)^{-\frac{1}{2}\left(\alpha -\frac{n}{r}\right)+i\nu}. \nonumber
\end{align}
Hence, 
\begin{align}
D_{\alpha}^{(1)}(e^{i(t|u)}\Delta(u)^{-\frac{1}{2}\left(\alpha -\frac{n}{r}\right)+i\nu})
&=\tr{(-u\nabla_{u}^{2}-\alpha\nabla_{u}+u-\alpha{e})}(e^{i(t|u)}\Delta(u)^{-\frac{1}{2}\left(\alpha -\frac{n}{r}\right)+i\nu}) \nonumber \\
&=\tr\!{\left(\!(e+t^{2})u+\left(2\nu -i\frac{n}{r}\right)t-\alpha e+\left(\frac{1}{4}\left(\alpha -\frac{n}{r}\right)^{2}+\nu^{2} \right)u^{-1}\!\right)} \nonumber \\
{} & \quad \cdot e^{i(t|u)}\Delta(u)^{-\frac{1}{2}\left(\alpha -\frac{n}{r}\right)+i\nu} \nonumber \\
&=\widetilde{D_{\alpha,\nu}^{(2)}}e^{i(t|u)}\Delta(u)^{-\frac{1}{2}\left(\alpha -\frac{n}{r}\right)+i\nu}. \nonumber 
\end{align}
Therefore, 
\begin{align}
(\mathcal{F}_{\alpha,\nu}^{-1}D_{\alpha}^{(1)}\psi_{\mathbf{m}}^{(\alpha)})(t)
&=(\widetilde{D_{\alpha,\nu}^{(2)}}e^{i(t|u)}\Delta(u)^{-\frac{1}{2}\left(\alpha -\frac{n}{r}\right)+i\nu}|\overline{\psi_{\mathbf{m}}^{(\alpha)}})_{L^{2}_{\alpha}(\Omega)} \nonumber \\
&=\widetilde{D_{\alpha,\nu}^{(2)}}(e^{i(t|u)}\Delta(u)^{-\frac{1}{2}\left(\alpha -\frac{n}{r}\right)+i\nu}|\overline{\psi_{\mathbf{m}}^{(\alpha)}})_{L^{2}_{\alpha}(\Omega)} \nonumber \\
&=\widetilde{D_{\alpha,\nu}^{(2)}}((\mathcal{F}_{\alpha,\nu}^{-1}\psi)(t)). \nonumber
\end{align}
The second equality is justified by $e^{i(t|u)}\Delta(u)^{\frac{1}{2}\left(\alpha -\frac{n}{r}\right)+i\nu}\psi_{\mathbf{m}}^{(\alpha)} \in L^{1}(\Omega)$. 
Finally, since $\mathcal{F}_{\alpha,\nu}^{-1}$ is an isomorphism from the space $L^{2}_{\alpha}(\Omega)^{K}$ onto $H_{\alpha, \nu}^{2}(V)^{K}$, 
we obtain $D_{\alpha,\nu}^{(2)}=\widetilde{D_{\alpha,\nu}^{(2)}}$.

On the other hand, for (\ref{eq:Differential equation for Psi 2}), from the definition of $D_{\alpha,\nu}^{(2)}$ and (\ref{eq:differential equation for Laguerre}), we have
$$
D_{\alpha,\nu}^{(2)}\Psi_{\mathbf{m}}^{(\alpha, \nu)}(t)
=D_{\alpha,\nu}^{(2)}\mathcal{F}_{\alpha,\nu}^{-1}(\psi_{\mathbf{m}}^{(\alpha)})(t)
=\mathcal{F}_{\alpha,\nu}^{-1}(D_{\alpha}^{(1)}\psi_{\mathbf{m}}^{(\alpha)})(t)
=2|\mathbf{m}|\mathcal{F}_{\alpha,\nu}^{-1}(\psi_{\mathbf{m}}^{(\alpha)})(t)
=2|\mathbf{m}|\Psi_{\mathbf{m}}^{(\alpha, \nu)}(t). 
$$
\end{proof}

We can also define $D_{\alpha,\nu}^{(3)}$ by the relation $D_{\alpha,\nu}^{(3)}\mathcal{C}_{\alpha,\nu}^{-1}=\mathcal{C}_{\alpha,\nu}^{-1}D_{\alpha,\nu}^{(2)}$. 
However, since there are difficulties in deriving the modified Cayley transform of $\nabla_{t}^{-1}$, 
we have not been able to obtain the explicit expression for $D_{\alpha,\nu}^{(3)}$ like that in the above theorem. 
On the other hand, when $\alpha =\frac{n}{r}, \nu =0$, $\Psi_{\mathbf{m}}^{(\frac{n}{r},0)}(t)$ becomes 
\begin{align}
\Psi_{\mathbf{m}}^{(\frac{n}{r},0)}(t)&=\Delta(e-it)^{-\frac{n}{r}}d_{\mathbf{m}}\sum_{\mathbf{k}\subset \mathbf{m}}(-1)^{|\mathbf{k}|}\binom{\mathbf{m}}{\mathbf{k}}\Phi_{\mathbf{k}}(2(e-it)^{-1}) \nonumber \\
&=\Delta(e-it)^{-\frac{n}{r}}d_{\mathbf{m}}\Phi_{\mathbf{m}}(c^{-1}(t)). \\
\end{align}
Further, the term of the pseudo-differential operator vanishes for $D_{\frac{n}{r},0}^{(2)}$. 
Hence, 
\begin{equation}
\label{eq:Differential equation for Psi special}
D_{\frac{n}{r},0}^{(2)}=-i\tr((e+t^{2})\partial_{t})-\frac{n}{r}\tr{(e+it)}.
\end{equation}
Therefore, we obtain the following explicit expression for $D_{\frac{n}{r},0}^{(3)}$ from the modified Cayley transform of $D_{\frac{n}{r},0}^{(2)}$.
\begin{equation}
D_{\frac{n}{r},0}^{(3)}=2\tr{(\sigma{\nabla_{\sigma}})}.
\end{equation}

\subsection{Determinant formulas}
In this subsection, we assume $d=2$ (In particular, we remark $\frac{n}{r}=r$). 
In the case $d=2$, the spherical polynomials $\Phi_{\mathbf{m}}$ degenerate to the Schur polynomials $s_{\mathbf{m}}$ (recall (\ref{eq:spherical and schur})). 
Further, there are some determinant formulas for $\Psi_{\mathbf{m}}^{(\alpha, \nu)}$ and $\phi_{\mathbf{m}}^{(\alpha,\nu)}$. 
Before giving main theorem, we provide some Lemmas needed to prove the determinant formulas. 
\begin{lem}[$\cite{H}$\,Theorem\,1.2.1]
\label{thm:det lem}
Consider $r$ power series of single variable $z \in \mathbb{C}$
$$
f_{j}(z)=\sum_{m\geq 0}A_{m}^{(j)}z^{j}\,\,\,\,(j=1,\cdots,r).
$$
Then, 
\begin{equation}
\label{eq:det lem}
\frac{\det{(f_{j}(z_{l}))}}{V(z_{1},\ldots,z_{r})}=\sum_{\mathbf{m} \in \mathscr{P}}A_{\mathbf{m}}s_{\mathbf{m}}(z_{1},\ldots,z_{r}), 
\end{equation}
where $V(z_{1},\ldots,z_{r})$ denote the Vandermonde determinant
$$
V(z_{1},\ldots,z_{r}):=\prod_{1\leq p<q\leq r}(z_{p}-z_{q}), 
$$
and 
$$
A_{\mathbf{m}}:=\det{(A_{m_{p}+r-p}^{(q)})}. 
$$
\end{lem}
\begin{cor}
For $w=\sum_{j=1}^{r}w_{j}c_{j}, z=\sum_{j=1}^{r}z_{j}c_{j}$ with $w_{1},\ldots,w_{r},z_{1},\ldots,z_{r} \in \mathbb{C}\setminus \{0\}$, 
\begin{equation}
\label{det express of Cauchy ker}
\Delta(w)^{-\alpha}\int_{K}\Delta(kw^{-1}-z)^{-\alpha}dk
=\delta !\prod_{j=1}^{r}\frac{1}{(\alpha-r+1)_{j-1}}\frac{\det((1-w_{p}z_{q})^{-(\alpha-r+1)})}{V(w_{1},\ldots,w_{r})V(z_{1},\ldots,z_{r})}.
\end{equation}
\end{cor}
\begin{proof}
We assume $|z_{j}|,|w_{j}|<1$ for $j=1,\ldots,r$ and consider the $r$ power series
$$
f_{j}(w)=(1-wz_{j})^{-(\alpha -r+1)}=\sum_{m\geq 0}\frac{(\alpha -r+1)_{m}}{m!}w^{m}z_{j}^{m}\,\,\,\,(j=1,\cdots,r). 
$$
From Lemma\,\ref{thm:det lem}, we have 
\begin{align}
\frac{\det((1-w_{p}z_{q})^{-(\alpha-r+1)})}{V(w_{1},\ldots,w_{r})V(z_{1},\ldots,z_{r})}
&=\sum_{\mathbf{m} \in \mathscr{P}}\det\left(\frac{(\alpha -r+1)_{m_{p}+r-p}}{(m_{p}+r-p)!}w_{q}^{m_{p}+r-p}\right)\frac{s_{\mathbf{m}}(z_{1},\ldots,z_{r})}{V(w_{1},\ldots,w_{r})} \nonumber \\
{} &=\sum_{\mathbf{m} \in \mathscr{P}}\left\{\prod_{j=1}^{r}\frac{(\alpha -r+1)_{m_{j}+r-j}}{m_{j}+r-j}\right\}s_{\mathbf{m}}(w_{1},\ldots,w_{r})s_{\mathbf{m}}(z_{1},\ldots,z_{r}) \nonumber \\
{} &=\frac{1}{\delta !}\prod_{j=1}^{r}(\alpha-r+1)_{j-1}
\sum_{\mathbf{m} \in \mathscr{P}}\frac{(\alpha)_{\mathbf{m}}}{\left(\frac{n}{r}\right)_{\mathbf{m}}}s_{\mathbf{m}}(w_{1},\ldots,w_{r})s_{\mathbf{m}}(z_{1},\ldots,z_{r}). \nonumber
\end{align}
Further, by (\ref{eq:d and Schur}), (\ref{eq:spherical and schur}) and (\ref{eq:Cauchy kernel of spherical poly}), we prove (\ref{det express of Cauchy ker}) for $|z_{1}|,\ldots,|z_{r}|,|w_{1}|,\ldots,|w_{r}|<1$. 
Finally, by analytic continuation, we prove this corollary holds for $w_{1},\ldots,w_{r},z_{1},\ldots,z_{r} \in \mathbb{C}\setminus \{0\}$.   
\end{proof}

\begin{thm}
{\rm{(1)}}\,Let $t=\sum_{j=1}^{r}t_{j}c_{j} \in V$. We have
\begin{equation}
\label{eq:det express of Psi}
\Psi_{\mathbf{m}}^{(\alpha, \nu)}(t)=\frac{s_{\mathbf{m}}(1,\ldots,1)}{(-2i)^{\frac{r(r-1)}{2}}}\delta !
\prod_{j=1}^{r}\frac{1}{\left(\frac{1}{2}(\alpha -r)+i\nu +1\right)_{j-1}}\frac{\det(\Psi_{m_{p}+r-p}^{(\alpha-r+1, \nu)}(t_{q}))}{V(t_{1},\ldots,t_{r})}.
\end{equation}

\noindent
{\rm{(2)}}\,Let $\sigma=\sum_{j=1}^{r}\sigma_{j}c_{j} \in \Sigma$. We have
\begin{equation}
\label{eq:det express of phi}
\phi_{\mathbf{m}}^{(\alpha,\nu)}(\sigma)=s_{\mathbf{m}}(1,\ldots,1)\delta !
\prod_{j=1}^{r}\frac{1}{\left(\frac{1}{2}(\alpha -r)+i\nu +1\right)_{j-1}}\frac{\det(\phi_{m_{p}+r-p}^{(\alpha-r+1, \nu)}(\sigma_{q}))}{V(\sigma_{1},\ldots,\sigma_{r})}.
\end{equation}
\end{thm}
\begin{proof}
{\rm{(1)}}\,We start from the generating formula for $\Psi_{\mathbf{m}}^{(\alpha, \nu)}(t)$. 
For $z=\sum_{j=1}^{r}z_{j}c_{j}, 0<z_{1},\ldots,z_{r}<1$, 
\begin{align}
\sum_{\mathbf{m} \in \mathscr{P}}\Psi_{\mathbf{m}}^{(\alpha,\nu)}(t)\Phi_{\mathbf{m}}(z)
&=\Delta(e-z)^{-\alpha}\int_{K}\Delta((e+z)(e-z)^{-1}-ikt)^{-\frac{1}{2}\left(\alpha +r\right)-i\nu}\,dk \nonumber \\
&=\prod_{j=1}^{r}(1-z_{j})^{-\alpha}\left(\frac{1+z_{j}}{1-z_{j}}\right)^{-\frac{1}{2}\left(\alpha +r\right)-i\nu} \nonumber \\
& \quad \cdot \delta !\prod_{j=1}^{r}\frac{1}{\left(\frac{1}{2}(\alpha -r)+i\nu +1\right)_{j-1}}
\frac{\det\left(\left(1-\frac{1-z_{p}}{1+z_{p}}it_{q}\right)^{-\frac{1}{2}(\alpha -r)-i\nu -1}\right)}{V\left(\frac{1-z_{1}}{1+z_{1}},\ldots,\frac{1-z_{r}}{1+z_{r}}\right)V(it_{1},\ldots,it_{r})}. \nonumber
\end{align}
The second equality follows from (\ref{det express of Cauchy ker}). 
Further, noticing that 
$$
\frac{1-z_{p}}{1+z_{p}}-\frac{1-z_{q}}{1+z_{q}}=-2\frac{z_{p}-z_{q}}{(1+z_{p})(1+z_{q})},
$$
we obtain
\begin{align}
\sum_{\mathbf{m} \in \mathscr{P}}\Psi_{\mathbf{m}}^{(\alpha,\nu)}(t)\Phi_{\mathbf{m}}(z)
&=(-2i)^{-\frac{r(r-1)}{2}}\delta !\prod_{j=1}^{r}\frac{1}{\left(\frac{1}{2}(\alpha -r)+i\nu +1\right)_{j-1}} \nonumber \\
& \quad \cdot 
\frac{\det\left((1-z_{p})^{-(\alpha-r+1)}\left(\frac{1+z_{p}}{1-z_{p}}-it_{q}\right)^{-\frac{1}{2}(\alpha -r+1)-\frac{1}{2}-i\nu}\right)}{V\left(z_{1},\ldots,z_{r}\right)V(t_{1},\ldots,t_{r})}. \nonumber
\end{align}
Here, by (\ref{eq:generating fnc of Psi}), we remark
$$
f_{q}(z)=(1-z)^{-(\alpha-r+1)}\left(\frac{1+z}{1-z}-it_{q}\right)^{-\frac{1}{2}(\alpha -r+1)-\frac{1}{2}-i\nu}
=\sum_{m\geq 0}\Psi_{m}^{(\alpha-r+1,\nu)}(t_{q})z^{m}. 
$$
Therefore, we expand the above determinant expression in Schur function series by using Lemma\,\ref{thm:det lem}.
\begin{align}
\sum_{\mathbf{m} \in \mathscr{P}}\Psi_{\mathbf{m}}^{(\alpha,\nu)}(t)\Phi_{\mathbf{m}}(z)
&=(-2i)^{-\frac{r(r-1)}{2}}\delta !\prod_{j=1}^{r}\frac{1}{\left(\frac{1}{2}(\alpha -r)+i\nu +1\right)_{j-1}} \nonumber \\
& \quad \cdot 
\sum_{\mathbf{m} \in \mathscr{P}}\frac{\det(\Psi_{m_{p}+r-p}^{(\alpha-r+1,\nu)}(t_{q}))}{V(t_{1},\ldots,t_{r})}s_{\mathbf{m}}(z_{1},\ldots,z_{r}). \nonumber
\end{align}
Finally, by comparing of $s_{\mathbf{m}}(z_{1},\ldots,z_{r})=s_{\mathbf{m}}(1,\ldots,1)\Phi_{\mathbf{m}}(z)$ on the above equation for $\mathbf{m} \in \mathscr{P}$, we obtain (\ref{eq:det express of Psi}).

\noindent
{\rm{(2)}}\,Applying the modified Cayley transform $\mathcal{C}_{\alpha,\nu}^{-1}$ to (\ref{eq:det express of Psi}), we obtain
\begin{align}
\phi_{\mathbf{m}}^{(\alpha,\nu)}(\sigma)
&=(-2i)^{-\frac{r(r-1)}{2}}s_{\mathbf{m}}(1,\ldots,1)\delta !\prod_{j=1}^{r}\frac{1}{\left(\frac{1}{2}(\alpha -r)+i\nu +1\right)_{j-1}} \nonumber \\
& \quad \cdot \prod_{j=1}^{r}\left(\frac{1-\sigma_{j}}{2}\right)^{-\frac{1}{2}(\alpha +r)-i\nu}\frac{\det\left(\Psi_{m_{p}+r-p}^{(\alpha-r+1, \nu)}\left(i\frac{1+\sigma_{q}}{1-\sigma_{q}}\right)\right)}{V\left(i\frac{1+\sigma_{1}}{1-\sigma_{1}},\ldots,i\frac{1+\sigma_{r}}{1-\sigma_{r}}\right)}. \nonumber
\end{align}
Since 
$$
i\frac{1+\sigma_{p}}{1-\sigma_{p}}-i\frac{1+\sigma_{q}}{1-\sigma_{q}}=2i\frac{\sigma_{p}-\sigma_{q}}{(1-\sigma_{p})(1-\sigma_{q})}
$$
and 
$$
\phi_{m_{p}+r-p}^{(\alpha-r+1, \nu)}(\sigma_{q})=\left(\frac{1-\sigma_{q}}{2}\right)^{-\frac{1}{2}((\alpha -r+1)+1)-i\nu}\Psi_{m_{p}+r-p}^{(\alpha-r+1, \nu)}\left(i\frac{1+\sigma_{q}}{1-\sigma_{q}}\right),
$$
we have 
\begin{align}
(-2i)^{-\frac{r(r-1)}{2}}\frac{\det\left(\Psi_{m_{p}+r-p}^{(\alpha-r+1, \nu)}\left(i\frac{1+\sigma_{q}}{1-\sigma_{q}}\right)\right)}{V\left(i\frac{1+\sigma_{1}}{1-\sigma_{1}},\ldots,i\frac{1+\sigma_{r}}{1-\sigma_{r}}\right)}
&=\prod_{j=1}^{r}\left(\frac{1-\sigma_{j}}{2}\right)^{\frac{1}{2}(\alpha +r)+i\nu}\frac{\det(\phi_{m_{p}+r-p}^{(\alpha-r+1, \nu)}(\sigma_{q}))}{V(\sigma_{1},\ldots,\sigma_{r})}. \nonumber
\end{align}
Therefore, we obtain the conclusion. 
\end{proof}

\subsection{One variable case}
In this subsection, we have assumed that $r=1$. 

First, we remark that (\ref{eq:def of MCJ}) becomes
\begin{align}
\phi_{m}^{(\alpha, \nu)}(\sigma)&:=\frac{(\alpha)_{m}}{m!}\sum_{k=0}^{m}(-1)^{k}\binom{m}{k}\frac{\left(\frac{1}{2}(\alpha +1)+i\nu \right)_{k}}{(\alpha)_{k}}(1-\sigma)^{k} \\
&=\frac{(\alpha)_{m}}{m!}{_{2}F_1}\left(\begin{matrix}-m,\frac{1}{2}(\alpha+1)+i\nu \\{\alpha} \end{matrix};1-\sigma \right) \\
&=\frac{\left(\frac{\alpha -1}{2}-i\nu \right)}{m!}{_{2}F_1}\left(\begin{matrix}-m,\frac{1}{2}(\alpha+1)+i\nu \\{-m-\frac{\alpha -3}{2}+i\nu } \end{matrix};\sigma \right).
\end{align}
and for $\alpha >0, \nu \in \mathbb{R}$, (\ref{eq:MCJ orthogonal}) degenerates to 
\begin{align}
\frac{1}{2\pi{i}}\int_{\Sigma} \phi_{m}^{(\alpha, \nu)}(\sigma)\overline{\phi_{n}^{(\alpha, \nu)}(\sigma)}
|(1-\sigma)^{\frac{\alpha-1}{2}+i\nu}|^{2}\,d\mu(\sigma)=\frac{\Gamma(\alpha+m)}{m!}\frac{1}{\left|\Gamma(\frac{\alpha+1}{2}+i\nu)\right|^{2}}\delta_{mn}.
\end{align}
That is a 1-parameter deformation of the usual circular Jacobi polynomial that coincides with $\phi_{m}^{(\alpha, 0)}(\sigma)$. 
In particular, $\phi_{m}^{(1, 0)}(\sigma)=\sigma^{m}$ and 
$$
\frac{1}{2\pi{i}}\int_{\Sigma} \phi_{m}^{(1, 0)}(\sigma)\overline{\phi_{n}^{(1, 0)}(\sigma)}\,d\mu(\sigma)
=\frac{1}{2\pi}\int_{0}^{2\pi}e^{im\theta}e^{-in\theta}\,d\theta=\delta_{mn}. 
$$
We also remark that the rank $1$ case of (\ref{eq:CJacobi and MP}) is 
\begin{equation}
\phi_{m}^{(\alpha, \nu)}(e^{i\theta})=q_{m}^{\left(\alpha, -\frac{\theta}{2}\right)}\left(\nu +\frac{1}{2i}\right)=e^{m\frac{i\theta}{2}}P_{m}^{\left(\frac{\alpha}{2}\right)}\left(\nu +\frac{1}{2i};-\frac{\theta}{2}\right),
\end{equation}
That means if $\theta$ is regarded as a parameter and $\nu$ is regarded as a variable for the circular Jacobi polynomial, 
then we can consider the circular Jacobi polynomial to be the Meixner-Pollaczek polynomial. 

Moreover, the generating function of $\phi_{m}^{(\alpha, \nu)}(\sigma)$ is given by
\begin{equation}
\sum_{m\geq 0}\phi_{m}^{(\alpha, \nu)}(\sigma)z^{m}=(1-z)^{-\frac{1}{2}(\alpha -1)+i\nu}(1-\sigma z)^{-\frac{1}{2}(\alpha +1)-i\nu}.
\end{equation}

Although, we have not been given an explicit expression for the differential relation of $\phi_{\mathbf{m}}^{(\alpha, \nu)}(\sigma)$ in the multivariate case, 
we obtain the following explicit result in the one variable case from the differential equation of ${_{2}F_1}$. 
\begin{prop}
\label{thm:one variable differential}
If
\begin{align}
D_{\alpha ,\nu }:=\sigma (1-\sigma )\partial_{\sigma}^{2}+\left\{\left(-m+\frac{3}{2}+i\nu \right)(1-\sigma)-\frac{\alpha}{2}(1+\sigma)\right\}\partial_{\sigma}+m\left(\frac{1}{2}(\alpha +1)+i\nu \right),
\end{align}
then
\begin{equation}
D_{\alpha ,\nu }\phi_{m}^{(\alpha, \nu)}(\sigma)=0.
\end{equation}
\end{prop}

\section{Concluding remarks}
We have investigated the fundamental properties of MCJ polynomials, that is, orthogonality and the generating function etc. 
However, as we have not succeeded in obtaining a differential equation for $\phi_{\mathbf{m}}^{(\alpha,\nu)}$ similar to Proposition\,\ref{thm:one variable differential}, 
we can not derive a modified Cayley transform of $\tr{\nabla_{u}^{-1}}$.

It is also important to consider the generalization of MCJ polynomials for multiplicity $d$. 
Actually, this generalization has been obtained by Baker and Forrester \cite{BF} for multivariate Laguerre polynomials which are a modified Fourier transform of the Cayley transform of MCJ polynomials. 
In addition, we can consider MCJ polynomials and their orthogonality without using the analysis on the symmetric cones as follows.

Let $n:=r+\frac{d}{2}r(r-1)$, 
\begin{align}
\label{eq:generalized d_{m}}
d_{\mathbf{m}}&:=
\prod_{j=1}^{r}\frac{\Gamma\left(\frac{d}{2}\right)}{\Gamma\left(\frac{d}{2}j\right)\Gamma\left(\frac{d}{2}(j-1)+1\right)} \nonumber \\
{} & \quad \cdot \prod_{1\leq p<q\leq r}\left(m_{p}-m_{q}+\frac{d}{2}(q-p)\right)\frac{\Gamma\left( m_{p}-m_{q}+\frac{d}{2}(q-p+1)\right)}{\Gamma\left(m_{p}-m_{q}+\frac{d}{2}(q-p-1)+1\right)}, \\
\Gamma_{\Omega}(\mathbf{s})&:=(2\pi)^{\frac{n-r}{2}}\prod_{j=1}^{r}\Gamma\left(s_{j}-\frac{d}{2}(j-1)\right), \nonumber \\
(\mathbf{s})_{\mathbf{k}}&:=\prod_{j=1}^{r}\left(s_{j}-\frac{d}{2}(j-1)\right)_{k_{j}}. \nonumber
\end{align}
Further, $P_{\mathbf{k}}^{(\frac{2}{d})}(\lambda_{1},\ldots,\lambda_{r})$ is an $r$-variable Jack polynomial 
and 
\begin{equation}
\label{eq:generalized spherical poly}
\Phi_{\mathbf{k}}^{(d)}(\lambda_{1},\ldots,\lambda_{r}):=\frac{P_{\mathbf{k}}^{(\frac{2}{d})}(\lambda_{1},\ldots,\lambda_{r})}{P_{\mathbf{k}}^{(\frac{2}{d})}(1,\ldots,1)}.
\end{equation} 
Furthermore, we introduce the generalized (Jack) binomial coefficients based on \cite{OO} by 
$$
\Phi_{\mathbf{m}}^{(d)}(1+\lambda_{1},\ldots,1+\lambda_{r})
=\sum_{\mathbf{k}\subset \mathbf{m}}\binom{\mathbf{m}}{\mathbf{k}}_{\frac{d}{2}}\Phi_{\mathbf{k}}^{(d)}(\lambda_{1},\ldots,\lambda_{r}).
$$
\begin{dfn}
We define the generalized MCJ (GMCJ) polynomial as follows. 
\begin{align}
\phi_{\mathbf{m}}^{(d\,;\,\alpha, \nu)}(e^{i\theta})=\phi_{\mathbf{m}}^{(d\,;\,\alpha, \nu)}(e^{i\theta_{1}},\ldots,e^{i\theta_{r}})&:=d_{\mathbf{m}}\frac{(\alpha)_{\mathbf{m}}}{\left(\frac{n}{r}\right)_{\mathbf{m}}}
\sum_{\mathbf{k}\subset \mathbf{m}}(-1)^{|\mathbf{k}|}\binom{\mathbf{m}}{\mathbf{k}}_{\frac{d}{2}}\frac{\left(\frac{1}{2}\left(\alpha +\frac{n}{r}\right)+i\nu \right)_{\mathbf{k}}}{(\alpha)_{\mathbf{k}}} \nonumber \\
{} & \quad \cdot \Phi_{\mathbf{k}}^{(d)}(1-e^{i\theta_{1}},\ldots,1-e^{i\theta_{r}}).
\end{align}
\end{dfn}
Therefore, we present the following conjecture. 
\begin{conj}
\label{thm:conj MCJ}
If $\alpha >\frac{n}{r}-1,\, \nu \in \mathbb{R},\, d>0$, then 
\begin{align}
&\frac{\widetilde{c_{0}}}{(2\pi)^{n}}\int_{\mathcal{S}^{r}}
\phi_{\mathbf{m}}^{(d\,;\,\alpha, \nu)}(e^{i\theta})\overline{\phi_{\mathbf{n}}^{(d\,;\,\alpha, \nu)}(e^{i\theta})}
\prod_{j=1}^{r}|(1-e^{i\theta_{j}})^{\frac{1}{2}\left(\alpha -\frac{n}{r}\right)+i\nu}|^{2}
\prod_{1\leq p<q\leq r}|e^{i\theta_{p}}-e^{i\theta_{q}}|^{d}\,d\theta_{1}\cdots d\theta_{r} \nonumber \\
&=d_{\mathbf{m}}\frac{\Gamma_{\Omega}(\alpha+\mathbf{m})}{\left(\frac{n}{r}\right)_{\mathbf{m}}}\frac{1}{\left|\Gamma_{\Omega}\left(\frac{1}{2}\left(\alpha +\frac{n}{r}\right)+i\nu\right)\right|^{2}}\delta_{\mathbf{m}\mathbf{n}}.
\end{align}
\end{conj}
For some special case, we prove the conjecture. 
\begin{prop}
{\rm{(1)}}\,If $d=1,2,4$ or $r=2, d \in \mathbb{Z}_{> 0}$ or $r=3, d=8$, then this conjecture is true.

\noindent
{\rm{(2)}}\,The case of $\alpha =\frac{n}{r}$ and $\nu =0$ is also true.  
\end{prop}
\begin{proof}
{\rm{(1)}}\,It follows immediately from Theorem\,\ref{thm:main theorem1} and the classification of irreducible symmetric cones. 

\noindent
{\rm{(2)}}\,We remark that when $\alpha =\frac{n}{r},\nu =0$, 
\begin{align}
\phi_{\mathbf{m}}^{\left(d\,;\,\frac{n}{r},0\right)}\left(e^{i\theta}\right)
&=d_{\mathbf{m}}\sum_{\mathbf{k}\subset \mathbf{m}}(-1)^{|\mathbf{k}|}\binom{\mathbf{m}}{\mathbf{k}}_{\frac{d}{2}}
\frac{P_{\mathbf{k}}^{(\frac{2}{d})}(1-e^{i\theta_{1}},\ldots,1-e^{i\theta_{r}})}{P_{\mathbf{k}}^{(\frac{2}{d})}(1,\ldots,1)} \nonumber \\
\label{eq:degenerate to Jack}
&=d_{\mathbf{m}}\frac{P_{\mathbf{m}}^{(\frac{2}{d})}(e^{i\theta_{1}},\ldots,e^{i\theta_{r}})}{P_{\mathbf{m}}^{(\frac{2}{d})}(1,\ldots,1)}.
\end{align}
For the Jack polynomial, the following formulas are known (see (6.4) in \cite{OO} and (10.38) in \cite{M} respectively). 
\begin{align}
\label{eq:Jack 1}
P_{\mathbf{m}}^{\left(\frac{2}{d}\right)}(1,\ldots,1)
&=\prod_{j=1}^{r}\frac{\Gamma\left(\frac{d}{2}\right)}{\Gamma\left(\frac{d}{2}j\right)}
\prod_{1\leq p<q\leq r}\frac{\Gamma\left(m_{p}-m_{q}+\frac{d}{2}(q-p+1)\right)}{\Gamma\left(m_{p}-m_{q}+\frac{d}{2}(q-p)\right)}, \\
\|P_{\mathbf{m}}^{\left(\frac{2}{d}\right)}\|_{r,\frac{2}{d}}^{2}
&:=\frac{1}{(2\pi)^{r}}\frac{1}{r!}\int_{\mathcal{S}^{r}}|P_{\mathbf{m}}^{\left(\frac{2}{d}\right)}(e^{i\theta_{1}},\ldots,e^{i\theta_{r}})|^{2}
\prod_{1\leq p<q\leq r}|e^{i\theta_{p}}-e^{i\theta_{q}}|^{d}\,d\theta_{1}\cdots d\theta_{r} \nonumber \\
\label{eq:norm of the Jack}
&=\prod_{1\leq p<q\leq r}\frac{\Gamma\left(m_{p}-m_{q}+\frac{d}{2}(q-p+1)\right)\Gamma\left(m_{p}-m_{q}+\frac{d}{2}(q-p-1)+1\right)}
{\Gamma\left(m_{p}-m_{q}+\frac{d}{2}(q-p)\right)\Gamma\left(m_{p}-m_{q}+\frac{d}{2}(q-p)+1)\right)}. 
\end{align}
Hence, from (\ref{eq:generalized d_{m}}), (\ref{eq:Jack 1}) and (\ref{eq:norm of the Jack}), we have 
\begin{equation}
d_{\mathbf{m}}\|P_{\mathbf{m}}^{\left(\frac{2}{d}\right)}\|_{r,\frac{2}{d}}^{2}
=\left\{\prod_{j=1}^{r}\frac{1}{\Gamma\left(\frac{d}{2}(j-1)+1\right)}\frac{\Gamma\left(\frac{d}{2}j\right)}{\Gamma\left(\frac{d}{2}\right)}\right\}
P_{\mathbf{m}}^{\left(\frac{2}{d}\right)}(1,\ldots,1)^{2}.
\end{equation}
Therefore, by (\ref{eq:degenerate to Jack}) and the orthogonality of the Jack polynomial, we obtain
\begin{align}
&\frac{\widetilde{c_{0}}}{(2\pi)^{n}}\int_{\mathcal{S}^{r}}
\phi_{\mathbf{m}}^{\left(d\,;\,\frac{n}{r},0\right)}\left(e^{i\theta}\right)\overline{\phi_{\mathbf{n}}^{\left(d\,;\,\frac{n}{r},0\right)}\left(e^{i\theta}\right)}
\prod_{1\leq p<q\leq r}|e^{i\theta_{p}}-e^{i\theta_{q}}|^{d}\,d\theta_{1}\cdots d\theta_{r} \nonumber \\
&=\frac{1}{(2\pi)^{n}}\frac{(2\pi)^{\frac{n-r}{2}}}{r!}
\left\{\prod_{j=1}^{r}\frac{\Gamma\left(\frac{d}{2}\right)}{\Gamma\left(\frac{d}{2}j\right)}\right\}
\frac{d_{\mathbf{m}}^{2}}{P_{\mathbf{m}}^{\left(\frac{2}{d}\right)}(1,\ldots,1)^{2}}(2\pi)^{r}r!\|P_{\mathbf{m}}^{\left(\frac{2}{d}\right)}\|_{r,\frac{2}{d}}^{2}\delta_{\mathbf{m}\mathbf{n}}  \nonumber \\
&=\frac{d_{\mathbf{m}}}{(2\pi)^{\frac{n-r}{2}}}\prod_{j=1}^{r-1}\frac{1}{\Gamma\left(1+\frac{d}{2}(r-j)\right)}\delta_{\mathbf{m}\mathbf{n}}
=d_{\mathbf{m}}\frac{1}{\Gamma_{\Omega}\left(\frac{n}{r}\right)}\delta_{\mathbf{m}\mathbf{n}}. \nonumber
\end{align}
\end{proof}
To summarize, we draw the following diagram of GMCJ polynomial. 
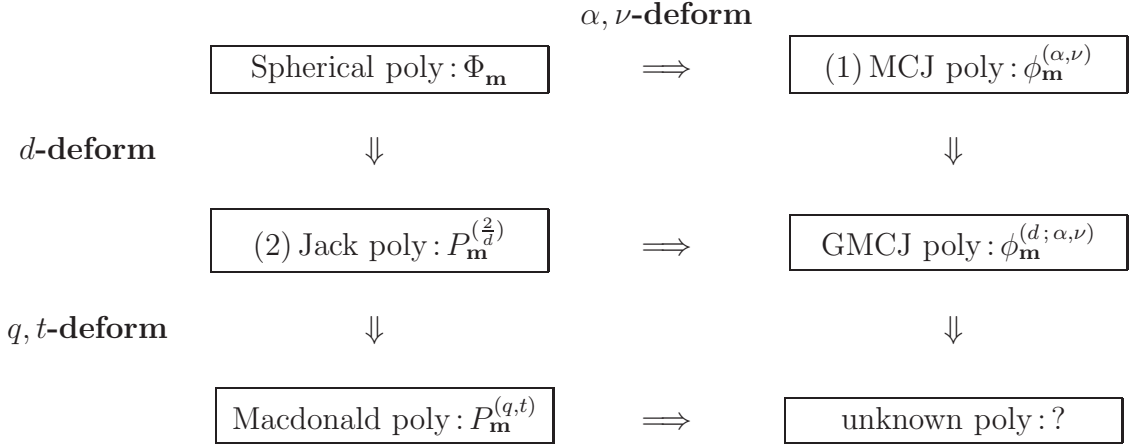
\begin{align}
\begin{array}{ccccc}
& &{\text{{\bf{$\alpha ,\nu$-deform}}}} & & \\
&{\text{
\framebox[4.5cm][c]{Spherical poly\,:\,$\Phi_{\mathbf{m}}^{}$}}} & \Longrightarrow 
& {\text{
\framebox[4.5cm][c]{{\rm{(1)}}\,MCJ poly\,:\,$\phi_{\mathbf{m}}^{(\alpha, \nu)}$}}} &\\
& & & & \\
{\text{{\bf{$d$-deform}}}}&\Downarrow & & \Downarrow &\\
& & & &\\
&{\text{
\framebox[4.5cm][c]{{\rm{(2)}}\,Jack poly\,:\,$P_{\mathbf{m}}^{(\frac{2}{d})}$}}} & \Longrightarrow 
& {\text{
\framebox[4.5cm][c]{GMCJ poly\,:\,$\phi_{\mathbf{m}}^{(d\,;\,\alpha, \nu)}$}}} &\\
& & & &\\
{\text{{\bf{$q,t$-deform}}}}&\Downarrow & & \Downarrow &\\
& & & & \\
&\,\,{\text{
\framebox[4.5cm][c]{Macdonald poly\,:\,$P_{\mathbf{m}}^{(q,t)}$}}} & \Longrightarrow & \framebox[4.5cm][c]{{\text{unknown poly\,:\,?}}} &\nonumber 
\end{array}
\end{align}


Although a proof of a general case would be desirable, 
our method in this paper cannot be applied to a general case. 
It may be necessary to consider a method of quantum integrable systems, 
that is, to construct some commuting families of differential or pseudo-differential operators whose simultaneous eigenfunctions become GMCJ polynomials. 
Actually, for the multivariate Laguerre polynomials, there exist a commuting family of some differential operators $D_{1},\ldots,D_{r}$ such that 
$$
D_{k}\psi_{\mathbf{m}}^{(\alpha)}(u)=\lambda_{k,\mathbf{m}}\psi_{\mathbf{m}}^{(\alpha)}(u),
$$ 
for all $k=1,\ldots,r$. 
Hence, if we put $\widetilde{D_{k}}:=(\mathcal{C}_{\alpha,\nu}^{-1}\circ\mathcal{F}_{\alpha,\nu}^{-1})\circ D_{k}\circ (\mathcal{F}_{\alpha,\nu}\circ \mathcal{C}_{\alpha,\nu})$, then these (pseudo?) differential operators $\widetilde{D_{k}}$ commute with each other and for all $k=1,\ldots,r$, 
$$
\widetilde{D_{k}}\phi_{\mathbf{m}}^{(d\,;\,\alpha, \nu)}(\sigma)=\lambda_{k,\mathbf{m}}\phi_{\mathbf{m}}^{(d\,;\,\alpha, \nu)}(\sigma). 
$$
That is to say, MCJ polynomials give a new quantum integrable system. 
Hence, although we are not aware of any studies of constructions of commuting families of pseudo-differential operators thus far for GMCJ polynomials, 
we think our conjecture is likely to be an important target of investigations into quantum integrable systems.

Therefore, we give the more problems for GMCJ polynomials, which are very related to the above conjecture.  
First, we remark commuting families of differential operators for the Jack polynomials are given by using a degenerate double affine Hecke algebra (DDAHA) of type A.  
Moreover, since GMCJ polynomials are 2-parameter deformation of the Jack polynomials, we present the following problem naturally. 
\begin{prob}
Discover a 2-parameter deformation of the DDAHA, which need to construct commuting families of pseudo-differential operators whose simultaneous eigenfunctions become GMCJ polynomials. 
\end{prob}
\noindent

\noindent
Next, we recall the following result for Jack polynomials. 
\begin{thm}
{\rm{(1)}}\,The singular vectors of the Virasoro algebra are Jack polynomials with rectangular partitions \cite{MY}.\\
\noindent
{\rm{(2)}}\,The singular vectors of the $\mathcal{W}$ algebra are Jack polynomials with arbitrary partitions \cite{AMOS}. 
\end{thm}
\noindent
From this theorem, 
we can also think the following problem. 
\begin{prob}
{\rm{(1)}}\,Discover a 2-parameter deformation of the Virasoro algebra whose singular vectors become the GMCJ polynomials with rectangular partitions.\\
\noindent
{\rm{(2)}}\,Discover a 2-parameter deformation of the $\mathcal{W}$ algebra whose singular vectors become the GMCJ polynomials with arbitrary partitions.
\end{prob}

Finally, we would like to raise the issue of applications for MCJ polynomials. 
In particular, since the weight function of the orthogonality relation for MCJ polynomials coincides with a circular Jacobi ensemble, 
we expect an application to the random matrix model whose density function is a circular Jacobi ensemble. 
However, further details on this are goals of future work.

\bibliographystyle{amsplain}

\noindent Institute of Mathematics for Industry, Kyushu University\\
744, Motooka, Nishi-ku, Fukuoka, 819-0395, JAPAN.\\
E-mail: g-shibukawa@math.kyushu-u.ac.jp

\end{document}